\documentclass[11pt]{amsart}
\usepackage{amsmath,amsthm,amssymb,enumerate,bbm}
\usepackage{euscript}
\usepackage{color}
\usepackage{dsfont}
\usepackage[left=2cm,right=2cm,top=3.5cm,bottom=3.5cm]{geometry}
\usepackage{hyperref}
\usepackage[capitalise]{cleveref}
\usepackage{mathrsfs}
\usepackage{graphicx}
\usepackage{tkz-euclide}
\usepackage[utf8]{inputenc}

\allowdisplaybreaks

\hypersetup{linkcolor=blue, colorlinks=true, citecolor = red}
\usepackage{upref}
\usepackage{url}
\usepackage[bottom]{footmisc}

\numberwithin{equation}{section}

\usepackage[maxbibnames=99,giveninits=true]{biblatex}
\bibliography{References.bib}

\newtheorem{Theorem}{Theorem}[section]
\newtheorem{Lemma}[Theorem]{Lemma}

\theoremstyle{definition}
\newtheorem{Definition}[Theorem]{Definition}
\newtheorem{Remark}[Theorem]{Remark}

\let\div\relax
\DeclareMathOperator{\div}{div}

\DeclareMathOperator{\supp}{supp}

\newcommand{\Id}{\operatorname{Id}}
\newcommand{\dd}{\,\mathrm{d}}
\newcommand{\del}{\partial}
\newcommand{\eps}{\varepsilon}

\newcommand{\ms}{\mathcal{S}}
\newcommand{\R}{\mathbb{R}}

\newcommand{\ba}{\boldsymbol{a}}
\newcommand{\bu}{\boldsymbol{u}}
\newcommand{\bv}{\boldsymbol{v}}
\newcommand{\bU}{\boldsymbol{U}}
\newcommand{\bUi}{\boldsymbol{U}_\infty}
\newcommand{\bV}{\boldsymbol{V}}
\newcommand{\bw}{\boldsymbol{u}}

\newcommand{\bS}{\mathbb S}

\newcommand{\vc}{\boldsymbol}
\newcommand{\bomega}{\boldsymbol{\omega}}
\newcommand{\ess}{\mathrm{ess}}
\newcommand{\res}{\mathrm{res}}
\newcommand{\esseps}{{\mathrm{ess},\eps}}
\newcommand{\reseps}{{\mathrm{res},\eps}}

\def\hmath$#1${\texorpdfstring{{\rmfamily\textit{#1}}}{#1}}

\title[Viscous compressible flow around a rotating body]{Weak-strong uniqueness and low Mach number limit for a viscous compressible fluid around a rotating body}

\author{Thomas Eiter}
\address{
Freie Universit\"at Berlin, Department of Mathematics and Computer Science, Arnimallee 14, 14195 Berlin, Germany, \textit{and}
Weierstrass Institute for Applied Analysis and Stochastics,
Anton-Wilhelm-Amo-Str.~39, 10117 Berlin, Germany.}
\email{thomas.eiter@wias-berlin.de}

\author{\v S\'arka Ne\v casov\'a}
\address{Institute of Mathematics, Czech Academy of Sciences,
\v Zitn\'a 25, 115 67 Praha 1, Czech Republic.}
\email{matus@math.cas.cz}

\author{Florian Oschmann}
\address{Faculty of Mathematics and Physics of the Charles University,
Sokolovsk\'a 49, 186 00 Praha 8, Czech Republic.}
 \email{florian.oschmann@matfyz.cuni.cz}
 
\date{}

\begin{document}

\begin{abstract}
We study the flow of an isothermal compressible Newtonian fluid around 
a body that performs a (time-independent) rigid motion.
We derive a weak-strong uniqueness principle,
and show that in the low Mach number limit, the governing equation is well approximated by the Navier--Stokes equations for incompressible rotating flow.
Both results are based on the derivation of a relative energy inequality for 
weak solutions to this exterior-domain problem.
\end{abstract}

\subjclass{%
76N06, 
76N10, 
76U05, 
76D05 
}

\keywords{Navier--Stokes equations, compressible flow, exterior domain, rotating body, weak-strong uniqueness, low Mach number limit, relative energy}

\maketitle

\tableofcontents



\section{Introduction}

We consider a rigid body $\ms$ that rotates with a fixed angular velocity $\bomega\in\R^3$ in the three-dimensional space filled 
with a compressible viscous fluid.
In a frame attached to the body,
we denote the fluid domain by $\Omega = \R^3 \setminus \overline{\ms}$, 
which is the complement of the rigid body.
The fluid flow in $\Omega$ is then governed by the Navier--Stokes equations
\begin{align}\label{NSE}
\left\{
\begin{aligned}
    \del_t \rho + \div(\rho (\bu - \bomega \times x)) &= 0 &&\mbox{in } (0,T)\times \Omega,\\
\del_t (\rho\bu) + \div ( \rho (\bu - \bomega \times x) \otimes \bu ) + \rho \bomega \times \bu  &=  \rho \vc f + \div \bS(\nabla \bu) -  \nabla p(\rho) && \mbox{in }(0,T)\times \Omega,\\
\bu &=  \bomega \times x && \mbox{on } (0,T)\times \del \Omega,\\
\rho(\cdot,x) \to \rho_\infty, \ \bu(\cdot,x) &\to \vc a_\infty && \mbox{as } |x| \to \infty,\\
\rho(0,\cdot)=\rho_0, \ (\rho \bu)(0,\cdot) &= \rho_0\bu_0 && \mbox{in } \Omega.
\end{aligned}
\right.
\end{align}
Here, $T>0$ is a given time horizon,
and $\rho\colon(0,T)\times\Omega\to[0,\infty)$ and $\bu\colon(0,T)\times\Omega\to\R^3$ 
denote the unknown fluid density and velocity, respectively.
We further prescribe an external force $\vc f\colon(0,T)\times\Omega\to\R^3$, and the density $\rho_\infty>0$ and the velocity $\ba_\infty\in\R^3$ 
at spatial infinity, where we assume that $\ba_\infty$ and $\bomega$ are parallel.
The viscous stress tensor $\bS(\nabla\bu)$ satisfies Newton's rheological law
\begin{equation}
\label{eq:stress}
    \bS(\nabla\bu) = \mu \Big( \nabla\bu + \nabla \bu^T - \frac23 \div \bu \Id \Big) + \eta \div \bu \Id,
\end{equation}
where  $\mu>0$ and $\eta \geq 0$ denote shear and bulk viscosity coefficients, respectively.
The pressure $p(\rho)$ satisfies
\begin{equation}\label{assPress}
    p \in C^1([0,\infty)) \cap C^2((0,\infty)), \quad p(0)=0, \quad p'(\rho)>0 \ (\rho > 0), \quad \liminf_{\rho \to \infty} \frac{p'(\rho)}{\rho^{\gamma-1}}>0
\end{equation}
for some $\gamma>1$. 
While the existence of weak solutions requires $\gamma>\frac{3}{2}$, 
see also Theorem~\ref{thm:existence} below, 
we merely have to assume $\gamma>\frac{6}{5}$ for the results established in this article, hence including a larger part of the physically relevant regime where $\gamma \in (1, \frac53]$.

The first two lines of~\eqref{NSE} describe mass conservation and momentum balance, respectively.
We assume that fluid particles are attached to the body,
which is reflected by no-slip boundary conditions.
The system is complemented by conditions at spatial infinity
as well as initial conditions $\rho_0$ and $\rho_0\bu_0$ for the density and the momentum, respectively.
For more details on the physical motivation, the derivation of the model, 
and the transformation into the body frame,
we refer the reader to \cite{kravcmar2014motion}. 
See also~\cite{G2} for the case of an incompressible fluid. 
Additional to the rotation of $\ms$, one could prescribe a time-independent 
translational motion
with velocity $\boldsymbol\tau\in\R^3$ parallel to 
$\bomega$.
By a simple change of frames, 
this is equivalent to considering~\eqref{NSE} with $\ba_\infty$ replaced with $\ba_\infty-\boldsymbol\tau$.
Therefore, this case is included in the setting treated here.

Concerning the mathematical theory of
compressible fluids,
the fundamental results on existence of global-in-time weak solutions in three dimensions 
were obtained by P.-L.~Lions \cite{Lions2}
in the barotropic case with $p(\rho)= a\rho^\gamma$, $a>0$, $\gamma\geq\frac{9}{5}$,
and by Feireisl, Novotn\'y, Petzeltov\'a~\cite{FeNoPe},
who provided an extension to exponents $\gamma>\frac{3}{2}$. 
In particular, the latter result includes the physically relevant case
$\gamma=\frac{5}{3}$ of a monoatomic gas.
While these works consider the flow in a bounded domain,
the results have been generalized to a plethora of other configurations in the recent years.
Several results and further references 
can be found in the monographs~\cite{FN, KMNPWK2025comprmovingdom, NovotnyStraskraba2004introcompr},
where, among other cases, the flow around a body at rest,
the temperature-dependent case,
and the flow in moving domains are considered, respectively.
Concerning the present case of the flow around a rigid body with prescribed translational and rotational motion,
the existence of weak solutions to~\eqref{NSE} was derived in~\cite{kravcmar2014motion}.
An existence result for a system describing self-propelled motion in an unbounded three-dimensional domain can be found in~\cite{MN}.

The uniqueness of weak solutions to the compressible Navier--Stokes equations 
is only known under restrictive assumptions, 
see~\cite{Hoff2006unique} for instance,
but it is still an open question in the general case, even in a bounded domain.
However, there are multiple results on weak-strong uniqueness,
that is, the property that any weak solution coincides with the (hypothetical) strong solution
with the same initial data
as long as the latter exists. 
The first result in this respect is due to Germain~\cite{Germain2011wsuNS},
where the considered weak solutions satisfied a regularity assumption
that could not be expected to hold in general.
In~\cite{feireisl2012relative, FeireislNovotnySun2011suitweaksol} this condition was omitted,
and the weak-strong uniqueness property was instead 
established for the class of weak solutions
constructed in~\cite{FeNoPe, Lions2}, 
which satisfy an energy inequality,
so-called \textit{finite-energy weak solutions}.
Since then, this result has been generalized in various ways, see e.g.~\cite{FeNo1, JessleJinNovotny2013} for the temperature-dependent case, or \cite{KremlnecPia, NeRa} for the case with moving rigid body in a bounded domain filled by a compressible barotropic fluid with homogeneous or inhomogeneous boundary data. For the case of a prescribed moving domain filled by a compressible fluid, the weak-strong uniqueness was shown in \cite{KMNPWK2025comprmovingdom}. 
One main result of the present article is to establish 
a weak-strong uniqueness principle 
for finite-energy weak solutions to~\eqref{NSE},
see Theorem~\ref{thm:wsu} below.

The weak-strong uniqueness results in all these works 
were derived from a \textit{relative energy inequality (REI)}.
In the present case, the relative energy is given by
\begin{equation}
\label{eq:relen}
    E(\rho, \bu | \sigma, \bV) = \frac12 \rho |\bu - \bV|^2 + P(\rho | \sigma)
\end{equation}
for suitable functions $(\rho,\bu)$ and $(\sigma, \bV)$,
where 
the absolute and relative internal energy are defined as
\begin{align}
    \label{eq:H.def}
    H(\rho) = \rho \int_1^\rho 
    \frac{p'(z)}{z^2} \dd z, &&
    P(\rho | \sigma) = H(\rho) - H'(\sigma)(\rho - \sigma) - H(\sigma),
\end{align}
respectively.
Due to $p'>0$, the function $H$ is strictly convex,
and we have $E(\rho, \bu | \sigma, \bV)=0$
if and only if $(\rho,\bu)=(\sigma,\bV)$.
In this way, the relative energy provides a natural distance measure.
We derive a REI that describes (an inequality for) 
the evolution of $E(\rho, \bu | \sigma, \bV)$
when $(\rho,\bu)$ is a solution to~\eqref{NSE}.
More precisely, we consider finite-energy weak solutions $(\rho,\bu)$ to~\eqref{NSE}
and sufficiently smooth comparison functions $(\sigma,\bV)$,
see Theorem~\ref{thm:rei} below.

Besides the weak-strong uniqueness principle,
we also use the REI to study 
the limit of low Mach number,
which means that the flow velocity is small compared to
the speed of sound,
which is the case in many real-world applications.
In this setting, the fluid is nearly incompressible,
and one would expect convergence towards 
the Navier--Stokes equations for incompressible fluids, 
in our framework given by
\begin{equation}\label{NSE-LM-Limit}
\left\{\begin{aligned}
\div \bU &= 0 && \mbox{in } (0,T)\times \Omega,\\
\rho_\infty \big( \del_t \bU + \div ( (\bU - \bomega \times x) \otimes \bU ) + \bomega \times \bU \big) &=  \rho_\infty \vc f  + \div \bS(\nabla \bU) - \nabla \Pi 
&& \mbox{in } (0,T)\times \Omega,\\
\bU &=  \bomega \times x && \mbox{on } (0,T)\times \del \Omega,\\
\bU(\cdot,x) &\to \vc a_\infty && \mbox{as } |x| \to \infty,\\
\bU(0,\cdot) &= \bU_0 && \mbox{in } \Omega,
\end{aligned}\right.
\end{equation}
where $\bU\colon(0,T)\times\Omega\to\R^3$ and $\Pi\colon(0,T)\times\Omega\to\R$ 
are the velocity and pressure field of the incompressible fluid flow, respectively.
A brief sketch of the proof of existence,
further properties like asymptotic behavior, and an overview of the state-of-the-art theory on system~\eqref{NSE-LM-Limit}
can be found in the recent monograph~\cite{NKNP}.

The mathematical study of the low Mach number limit for systems of equations describing a motion of fluids goes back to the seminal work of Klainerman and Majda \cite{KM1}. In general, studying various types of singular limits allows to eliminate unimportant or unwanted features of the motion as a consequence of scaling and asymptotic analysis. 
The mathematical analysis of singular limits in the frame of both strong and weak solutions was carried out in \cite{Gallag, SCH2} and \cite{DGLM,FN}, respectively. 
We show that in the low Mach number limit,
solutions $(\rho,\bu)$ to~\eqref{NSE}
can be approximated by $(\rho_\infty,\bU)$, 
where $(\bU,\Pi)$ solves~\eqref{NSE-LM-Limit},
see Theorem~\ref{thm:singLim} below.
The analogous question in the case of pure rotation
and with a gravitational force, partial slip boundary conditions, and additional low Froude number limit (strong stratification)
was studied in~\cite{Feireisl2012lowMachrotating} using compactness arguments.
In contrast to~\cite{Feireisl2012lowMachrotating},
we only consider so-called well-prepared initial data here,
which prevents the occurrence of acoustic waves, but we allow the fluid not to be at rest when $|x| \to \infty$.\\

\paragraph{\textbf{Outline of the paper}}
In Section~\ref{sec:weaksol} we introduce the 
notion of weak solutions
and we recall the current existence results.
We then derive a REI for these weak solutions in
Section~\ref{sec:relen}.
This result is then applied to prove a weak-strong uniqueness principle
in Section~\ref{sec:wsu},
and to study the limit of vanishing Mach number in Section~\ref{sec:lowMach}.\\

\paragraph{\textbf{General assumptions}}
Throughout the whole article, we make the following assumptions:
Let $\ms\subset\R^3$ be a bounded domain 
and $\Omega = \R^3 \setminus \overline{\ms}$ be an exterior domain
with Lipschitz boundary $\del\Omega=\del\ms$.
Without loss of generality, we assume $0 \in {\rm int} \, \ms$ and $\overline\ms \subset B_{1/2}(0)$.
We assume that $\bomega,\ba_\infty \in\R^3$ satisfy 
$\bomega \times \vc a_\infty = 0$, and that $\rho_\infty>0$.
The viscous stress and the pressure are given by~\eqref{eq:stress} and~\eqref{assPress}
with viscosity coefficients $\mu>0$ and $\eta\geq 0$,
and with $\gamma>\frac{6}{5}$. 
We consider an external force that satisfies 
$\vc f \in L^\infty(0,T; [L^1 \cap L^\infty] (\Omega))$.
For the initital data 
$(\rho_0, \bu_0)\in L^1_{\mathrm{loc}}(\Omega)\times L^1_{\mathrm{loc}}(\Omega;\R^3)$ 
we assume that
\[
\rho_0 \geq 0,\qquad E(\rho_0, \bu_0 | \rho_\infty, \ba_\infty) \in L^1(\Omega),
\]
that is, finiteness of the initial (relative) energy.

\section{Weak solutions}
\label{sec:weaksol}

We introduce the class of weak solutions to~\eqref{NSE} 
satisfying an energy inequality. 
Since we study the problem in an unbounded domain,
the formulation of the latter requires the consideration of a reference velocity $\bU_\infty:\R^3\to\R^3$ satisfying
\begin{equation}\label{Uinfty}
\bU_\infty\in C^\infty(\R^3;\R^3),
\qquad 
\bU_\infty(x)= \bomega\times x \quad \text{for }x\in\ms,
\qquad
\bU_\infty(x)= \ba_\infty \quad\text{for }x\in\R^3\setminus B_1.
\end{equation}
Such a vector field can be constructed by standard methods. Observe that $\div \bU_\infty = 0$ in $\R^3 \setminus (B_1 \setminus \ms)$.
Moreover, we have 
$E(\rho, \bu | \rho_\infty, \ba_\infty) \in L^1(\Omega)$
if and only if
$E(\rho, \bu | \rho_\infty, \bU_\infty) \in L^1(\Omega)$.

\begin{Definition}\label{def:wksol}
We call a couple $(\rho, \bu)$ a \emph{finite-energy weak solution} to system \eqref{NSE} with initial data $(\rho_0, \bu_0)$ if it satisfies:
\begin{itemize}
\item the following regularity assumptions:
\begin{equation}
\label{eq:reg.weak}
\begin{gathered}
    \rho \geq 0, \quad \rho - \rho_\infty \in L^\infty(0,T; [L^\gamma + L^2] (\Omega)), \\
    \rho |\bu - \vc a_\infty|^2 \in L^\infty(0,T; L^1(\Omega)),\\
    \bu - \vc a_\infty \in L^2(0,T; W^{1,2}(\Omega)), \\
    \bu(t) |_{\del \Omega} = \bomega \times x \text{ for a.e. } t \in (0,T).
\end{gathered}
\end{equation}

\item the weak form of the continuity equation: we have $\rho \in C_{weak}([0,T]; L^\gamma(K))$ for any compact $K \subset \overline{\Omega}$, and
\begin{equation}
\label{eq:cont.weak}
\int_0^T \int_\Omega \rho  \del_t \phi + \rho (\bu - \bomega \times x) \cdot \nabla \phi \dd x \dd t = - \int_\Omega \rho_0 \,\phi(0, \cdot) \dd x
\end{equation}
for any $\phi \in C_c^1([0,T) \times \overline{\Omega})$.
\item the weak form of the momentum equation: we have $\rho \bu \in C_{weak}([0,T]; L^\frac{2\gamma}{\gamma+1}(K))$ for any compact $K \subset \overline{\Omega}$, and
\begin{equation}
\label{eq:mom.weak}
\begin{aligned}
\int_0^T \int_\Omega \rho \bu \cdot \del_t \psi 
&+ \rho ((\bu - \bomega \times x) \otimes \bu) : \nabla \psi - \rho (\bomega \times \bu) \cdot \psi\dd x \dd t 
\\
& + \int_0^T \int_\Omega p(\rho) \div \psi - \bS(\nabla \bu):\nabla \psi - \rho \vc f \cdot \psi \dd x \dd t 
= - \int_\Omega \rho_0 \bu_0 \cdot \psi(0, \cdot) \dd x
\end{aligned}
\end{equation}
for any $\psi \in C_c^1([0,T) \times \Omega; \R^3)$.
\item 
the energy inequality:
there exists a reference velocity $\bU_\infty:\R^3\to\R^3$ with~\eqref{Uinfty}
such that for a.a.~$\tau\in[0,T]$ it holds
\begin{equation}\label{enIneq}
    \begin{aligned}
    &\left[ \int_\Omega \frac12 \rho |\bu - \bU_\infty|^2 + P(\rho | \rho_\infty) \dd x \right]_{t=0}^{t=\tau} + \int_0^\tau \int_\Omega \bS(\nabla \bu) : \nabla (\bu - \bU_\infty) \dd x \dd t \\
    &\leq - \int_0^\tau \int_\Omega p(\rho) \div \bU_\infty \dd x \dd t - \int_0^\tau \int_\Omega \rho (\bu - \bomega \times x) \cdot \nabla \bU_\infty \cdot (\bu - \bU_\infty) \dd x \dd t \\
    &\qquad - \int_0^\tau \int_\Omega \rho (\bomega \times \bU_\infty) \cdot (\bu - \bU_\infty) \dd x \dd t + \int_0^\tau \int_\Omega \rho \vc f \cdot (\bu - \bU_\infty) \dd x \dd t.
    \end{aligned}
\end{equation}
\end{itemize}
\end{Definition}

\begin{Remark}
At first glace,
the energy inequality~\eqref{enIneq} only yields a bound on 
$\nabla\bu=\nabla(\bu-\ba_\infty)$ in $L^2((0,T)\times\Omega)$,
but not on the function $\bu-\ba_\infty$ itself.
However, using the finiteness of the total energy
and the Poincar\'e inequality from~\cite[Lemma 3.1]{JessleJinNovotny2013},
one can conclude $\bu - \vc a_\infty \in L^2(0,T; W^{1,2}(\Omega))$.
\end{Remark}

We have the following existence result:

\begin{Theorem}
\label{thm:existence}
    Let $\gamma>\frac{3}{2}$.
    Then there exists a finite-energy weak solution to system \eqref{NSE} in the sense of Definition~\ref{def:wksol}.
\end{Theorem}
\begin{proof}
    Weak solutions to~\eqref{NSE} were constructed in~\cite{kravcmar2014motion} for the case $\vc f=0$. However, it is evident that the same proof applies in the case $\vc f \in L^\infty(0,T; [L^1 \cap L^\infty] (\Omega))$.
    For the validity of the energy inequality~\eqref{enIneq}, see also~\cite[Remark~6.1]{kravcmar2014motion}.
\end{proof}

As mentioned above, 
although we cannot ensure the existence of solutions if $\gamma\leq \frac{3}{2}$,
all arguments in the subsequent sections will only make use of 
the weaker assumption $\gamma>\frac{6}{5}$.

\section{Relative energy inequality}
\label{sec:relen}

In this section, we derive a REI for weak solutions in the sense of Defintion~\ref{def:wksol}.
Recall the definition of the relative energy
$E(\rho, \bu | \sigma, \bV)$ from~\eqref{eq:relen}.

\begin{Theorem}
    \label{thm:rei}
    Let $(\rho,\bu)$ be a finite-energy weak solution to~\eqref{NSE} in the sense of Definition~\ref{def:wksol}.
    Let $0 < \underline{\sigma} \leq \rho_\infty \leq \overline{\sigma} < \infty$, 
    and let $(\sigma,\bV)\in C^1((0,T)\times\Omega)$ be in the following regularity class:
    \begin{equation}
    \label{regStr}
    \begin{gathered}
    \sigma - \rho_\infty \in L^\infty(0,T; [L^2 \cap L^{\frac{\gamma}{\gamma-1}}](\Omega)), \qquad 
    \underline{\sigma} \leq \sigma \leq \overline{\sigma},
    \\
    \nabla \sigma \in L^1(0,T; [L^2 \cap L^4 \cap L^\frac{2\gamma}{\gamma-1}] (\Omega)), \qquad 
    \del_t \sigma \in L^1(0,T; [L^2 \cap L^{\frac{\gamma}{\gamma-1}}] (\Omega)),
    \\
    (\ba_\infty-\bomega\times x)\cdot\nabla \sigma \in L^1(0,T; [L^1 \cap L^2 \cap L^{\frac{\gamma}{\gamma-1}}] (\Omega)),
    \\
    \bV - \vc a_\infty \in L^\infty(0,T; [L^2 \cap L^4 \cap L^{\frac{2\gamma}{\gamma-1}}] (\Omega)), 
    \\
    \nabla \bV \in L^1(0,T; [L^\frac{\gamma}{\gamma-1} \cap L^\infty] (\Omega)) \cap L^2((0,T)\times \Omega), 
    \\
    (\ba_\infty-\bomega\times x)\cdot\nabla\bV\in L^1(0,T; [L^2 \cap L^4 \cap L^{\frac{2\gamma}{\gamma-1}}] (\Omega)),
    \\
    \del_t \bV \in L^1(0,T; [L^2 \cap L^4 \cap L^{\frac{2\gamma}{\gamma-1}}] (\Omega)).
    \end{gathered}
    \end{equation}
Further assume that 
$\bV = \bomega \times x$ on $(0,T)\times\partial\Omega$.
Then 
\begin{equation}\label{REI}
\begin{aligned}
     &\left[ \int_\Omega E(\rho, \bu | \sigma, \bV)  \dd x \right]_{t=0}^{t=\tau} + \int_0^\tau \int_\Omega \bS(\nabla (\bu - \bV)) : \nabla (\bu - \bV) \dd x \dd t \\
    &\leq 
    \int_0^\tau \int_\Omega \rho (\bV - \bu) \cdot ( \del_t \bV  + \bomega\times\bV  + ((\bu-\bomega\times x) \cdot \nabla) \bV  - \vc f ) \dd x \dd t \\
    &\qquad - \int_0^\tau \int_\Omega \bS(\nabla \bV) : \nabla (\bu - \bV) \dd x \dd t  - \int_0^\tau \int_\Omega ( p(\rho) - p(\sigma)) \div \bV \dd x \dd t \\
    &\qquad - \int_0^\tau \int_\Omega (\rho - \sigma) \del_t H'(\sigma) + (\rho (\bu-\bomega\times x) - \sigma (\bV - \bomega \times x)) \cdot \nabla H'(\sigma) \dd x \dd t
\end{aligned}
\end{equation}
for a.a.~$\tau\in(0,T)$.
\end{Theorem}

\begin{Remark}
To see where the assumptions \eqref{regStr} come from, let us rewrite the REI \eqref{REI} in the following form, where we also used that $\bomega \times \vc a_\infty = 0$:
\begin{align*}
    &\left[ \int_\Omega E(\rho, \bu | \sigma, \bV)  \dd x \right]_{t=0}^{t=\tau} + \int_0^\tau \int_\Omega \bS(\nabla (\bu - \bV)) : \nabla (\bu - \bV) \dd x \dd t \\
    &\leq \int_0^\tau \int_\Omega \rho (\bV - \vc a_\infty) \cdot \del_t \bV \dd x \dd t + \int_0^\tau \int_\Omega \rho (\vc a_\infty - \bu) \cdot \del_t \bV \dd x \dd t \\
    &\quad + \int_0^\tau \int_\Omega \rho (\bV - \vc a_\infty) \cdot (  \bomega \times (\bV - \vc a_\infty)   ) \dd x \dd t + \int_0^\tau \int_\Omega \rho (\vc a_\infty - \bu) \cdot (  \bomega \times (\bV - \vc a_\infty)   ) \dd x \dd t \\
    &\quad + \int_0^\tau \int_\Omega (\bV - \vc a_\infty) \cdot (\rho(\bu - \vc a_\infty) \cdot \nabla) \bV \dd x \dd t + \int_0^\tau \int_\Omega \rho (\vc a_\infty - \bu) \cdot ((\bu - \vc a_\infty) \cdot \nabla) \bV \dd x \dd t \\
    &\quad + \int_0^\tau \int_\Omega \rho (\bV - \vc a_\infty) \cdot ((\vc a_\infty - \bomega\times x) \cdot \nabla) \bV \dd x \dd t + \int_0^\tau \int_\Omega \rho (\vc a_\infty - \bu) \cdot ((\vc a_\infty - \bomega\times x) \cdot \nabla) \bV \dd x \dd t \\
    &\quad - \int_0^\tau \int_\Omega \rho (\bV - \vc a_\infty) \cdot \vc f \dd x \dd t - \int_0^\tau \int_\Omega \rho (\vc a_\infty - \bu) \cdot \vc f \dd x \dd t \\
    &\quad - \int_0^\tau \int_\Omega \bS(\nabla \bV) : \nabla (\bu - \vc a_\infty) \dd x \dd t - \int_0^\tau \int_\Omega \bS(\nabla \bV) : \nabla (\vc a_\infty - \bV) \dd x \dd t \\
    &\quad - \int_0^\tau \int_\Omega ( p(\rho) - p'(\rho_\infty)(\rho - \rho_\infty) - p(\rho_\infty) ) \div \bV \dd x \dd t \\
    &\quad - \int_0^\tau \int_\Omega p'(\rho_\infty)[ (\rho - \rho_\infty) - (\sigma - \rho_\infty) ] \div \bV \dd x \dd t \\
    &\quad + \int_0^\tau \int_\Omega ( p(\sigma) - p'(\rho_\infty)(\sigma-\rho_\infty) - p(\rho_\infty) ) \div \bV \dd x \dd t \\
    &\quad - \int_0^\tau \int_\Omega (\rho - \rho_\infty) H''(\sigma) \del_t \sigma \dd x \dd t - \int_0^\tau \int_\Omega (\rho_\infty - \sigma) H''(\sigma) \del_t \sigma \dd x \dd t \\
    &\quad - \int_0^\tau \int_\Omega \rho (\bu- \vc a_\infty) \cdot H''(\sigma) \nabla \sigma \dd x \dd t \\
    &\quad - \int_0^\tau \int_\Omega (\rho - \rho_\infty) (\vc a_\infty-\bomega\times x) \cdot H''(\sigma) \nabla \sigma \dd x \dd t - \int_0^\tau \int_\Omega \rho_\infty (\vc a_\infty-\bomega\times x) \cdot H''(\sigma) \nabla \sigma \dd x \dd t \\
    &\quad + \int_0^\tau \int_\Omega \sigma (\bV - \vc a_\infty) \cdot H''(\sigma) \nabla \sigma \dd x \dd t + \int_0^\tau \int_\Omega \sigma (\vc a_\infty - \bomega \times x) \cdot H''(\sigma) \nabla \sigma \dd x \dd t .
\end{align*}
Using the regularity of $(\rho, \bu)$ mentioned in \eqref{eq:reg.weak}, which also gives us $\sqrt{\rho}\in L^\infty(L^{2\gamma}+L^4+L^\infty)$ and hence
\begin{align*}
    \rho (\bu - \vc a_\infty) = \sqrt{\rho}\sqrt{\rho} (\bu - \vc a_\infty) 
    \in L^\infty (L^{\frac{2\gamma}{\gamma+1}}+L^{4/3}+L^2),
\end{align*}
leads to the regularity from \eqref{regStr},
where we also use the general facts $(L^p \cap L^q)' = L^{p'} + L^{q'}$ and $L^p \cap L^q \subset L^p \subset L^p + L^q$.
\end{Remark}

In the case $\bomega = 0$, formula~\eqref{REI} corresponds to the usual REI  as can be found, for instance, in \cite{feireisl2012relative}. Note also that this inequality is independent of $\ba_\infty$ and thus of the choice of $\bU_\infty$. 

To prove Theorem~\ref{thm:rei}, we first consider the case that $(\sigma,\bV)$ differs from the limit values $(\rho_\infty,\ba_\infty)$ only on a compact set.

\begin{Lemma}
    \label{lem:rei.compact}
    In addition to the assumptions from Theorem~\ref{thm:rei}, assume that 
    that there exists $R>1$ such that 
    $\sigma(t,x)= \rho_\infty$ and $\bV(t,x) = \vc a_\infty$ for $t\in(0,T)$ and $|x|>R$.
    Then~\eqref{REI} holds for a.a.~$\tau\in(0,T)$.
\end{Lemma}

\begin{proof}
    
We first use the momentum equation~\eqref{eq:mom.weak} tested by $\bV - \bU_\infty$ to see
\begin{align*}
    &\left[ \int_\Omega \rho \bu \cdot (\bV - \bU_\infty) \dd x \right]_{t=0}^{t=\tau} = \int_0^\tau \int_\Omega \rho \bu \cdot \del_t \bV \dd x \dd t 
    - \int_0^\tau \int_\Omega \rho (\bomega\times\bu)\cdot (\bV-\bU_\infty) \dd x \dd t
    \\
    &\qquad+ \int_0^\tau \int_\Omega \rho (\bu-\bomega\times x) \otimes \bu : \nabla (\bV - \bU_\infty) \dd x \dd t - \int_0^\tau \int_\Omega \bS(\nabla \bu) : \nabla (\bV - \bU_\infty) \dd x \dd t \\
    &\qquad+ \int_0^\tau \int_\Omega p(\rho) \div(\bV - \bU_\infty) \dd x \dd t + \int_0^\tau \int_\Omega \rho \vc f \cdot (\bV - \bU_\infty) \dd x \dd t.
\end{align*}
Note that this choice of test function is admissible due to the definition of $\bU_\infty$ in \eqref{Uinfty} and the requirement that $\bV = \vc a_\infty$ for $|x| > R$.
Similarly, we can use $\frac12 |\bV - \bU_\infty|^2$ as test function in the continuity equation~\eqref{eq:cont.weak}, which yields
\begin{align*}
    \left[ \int_\Omega \frac12 \rho |\bV - \bU_\infty|^2 \dd x \right]_{t=0}^{t=\tau} 
    &= \int_0^\tau \int_\Omega \rho (\bV - \bU_\infty) \cdot \del_t \bV \dd x \dd t \\
    &\quad+ \int_0^\tau \int_\Omega \rho ((\bu-\bomega\times x) \cdot \nabla) (\bV - \bU_\infty) \cdot (\bV - \bU_\infty) \dd x \dd t.
\end{align*}
By the identity
\[
    |\bu - \bV|^2 = |\bu - \bU_\infty|^2 + |\bV - \bU_\infty|^2 - 2 (\bu - \bU_\infty) \cdot ( \bV - \bU_\infty)
\]
and the energy inequality \eqref{enIneq}, we infer
\begin{align*}
    &\left[ \int_\Omega \frac12 \rho |\bu - \bV|^2 + P(\rho | \rho_\infty) \dd x \right]_{t=0}^{t=\tau} + \int_0^\tau \int_\Omega \bS(\nabla \bu) : \nabla (\bu - \bV) \dd x \dd t \\
    &= \left[ \int_\Omega \frac12 \rho |\bu - \bU_\infty|^2 + P(\rho | \rho_\infty) \dd x \right]_{t=0}^{t=\tau} + \int_0^\tau \int_\Omega \bS(\nabla \bu ) : \nabla (\bu - \bU_\infty) \dd x \dd t \\
    &\qquad + \left[ \int_\Omega \frac12 \rho |\bV - \bU_\infty|^2 - \rho (\bu - \bU_\infty) \cdot (\bV - \bU_\infty) \dd x \right]_{t=0}^{t=\tau} + \int_0^\tau \int_\Omega \bS(\nabla \bu) : \nabla (\bU_\infty - \bV) \dd x \dd t\\
    &\leq - \int_0^\tau \int_\Omega p(\rho) \div \bU_\infty \dd x \dd t - \int_0^\tau \int_\Omega \rho (\bu - \bomega \times x) \cdot \nabla \bU_\infty \cdot (\bu - \bU_\infty) \dd x \dd t \\
    &\qquad - \int_0^\tau \int_\Omega \rho (\bomega \times \bU_\infty) \cdot (\bu - \bU_\infty) \dd x \dd t + \int_0^\tau \int_\Omega \rho \vc f \cdot (\bu - \bU_\infty) \dd x \dd t \\
    &\qquad + \int_0^\tau \int_\Omega \rho (\bV - \bU_\infty) \cdot \del_t \bV \dd x \dd t + \int_0^\tau \int_\Omega \rho (\bu-\bomega\times x) \cdot \nabla (\bV - \bU_\infty) \cdot (\bV - \bU_\infty) \dd x \dd t \\
    &\qquad - \int_0^\tau \int_\Omega \rho \bu \cdot \del_t \bV \dd x \dd t 
    + \int_0^\tau \int_\Omega \rho (\bomega\times\bu)\cdot (\bV-\bU_\infty) \dd x \dd t\\
    &\qquad - \int_0^\tau \int_\Omega \rho (\bu-\bomega\times x) \otimes \bu : \nabla (\bV - \bU_\infty) \dd x \dd t + \int_0^\tau \int_\Omega \bS(\nabla \bu) : \nabla (\bV - \bU_\infty) \dd x \dd t \\
    &\qquad - \int_0^\tau \int_\Omega p(\rho) \div(\bV - \bU_\infty) \dd x \dd t - \int_0^\tau \int_\Omega \rho \vc f \cdot (\bV - \bU_\infty) \dd x \dd t \\
    &\qquad + \left[ \int_\Omega \rho \bU_\infty \cdot (\bV - \bU_\infty) \dd x \right]_{t=0}^{t=\tau} 
    + \int_0^\tau \int_\Omega \bS(\nabla\bu) : \nabla (\bU_\infty - \bV) \dd x \dd t \\
    &= - \int_0^\tau \int_\Omega \rho (\bu - \bomega \times x) \cdot \nabla \bU_\infty \cdot (\bu - \bU_\infty) \dd x \dd t \\
    &\qquad - \int_0^\tau \int_\Omega \rho (\bomega\times\bV)\cdot(\bu-\bV) \dd x \dd t + \int_0^\tau \int_\Omega \rho \vc f \cdot (\bu - \bV) \dd x \dd t \\
    &\qquad + \int_0^\tau \int_\Omega \rho (\bV - \bU_\infty - \bu) \cdot (\del_t \bV + ((\bu-\bomega\times x) \cdot \nabla) (\bV - \bU_\infty) ) \dd x \dd t \\
    &\qquad - \int_0^\tau \int_\Omega p(\rho) \div \bV \dd x \dd t + \left[ \int_\Omega \rho \bU_\infty \cdot (\bV - \bU_\infty) \dd x \right]_{t=0}^{t=\tau} ,
\end{align*}
where we additionally used that
\begin{align*}
    (\bomega\times\bu)\cdot (\bV-\bUi) - (\bomega \times \bUi) \cdot (\bu-\bUi) &= (\bomega\times\bu)\cdot (\bV-\bUi) + (\bomega \times \bu)\cdot\bUi \\
    &= -(\bomega\times\bV)\cdot(\bV-\bU).
\end{align*}
Lastly, we generate the term $P(\rho | \rho_\infty)$ that shall replace $P(\rho | \sigma)$. To this end, 
recall~\eqref{eq:H.def} and find
\[
P(\rho | \sigma) - P(\rho | \rho_\infty) 
= - \rho (H'(\sigma) - H'(\rho_\infty)) 
+ \sigma H'(\sigma) - H(\sigma) - \rho_\infty H'(\rho_\infty) + H(\rho_\infty).
\]
With the identity
\begin{align*}
    &\left[ \int_\Omega \sigma H'(\sigma) - H(\sigma) - \rho_\infty H'(\rho_\infty) + H(\rho_\infty) \dd x \right]_{t=0}^{t=\tau}
    \\
    &=\int_0^\tau \frac{\dd}{\dd t} \int_\Omega \sigma H'(\sigma) - H(\sigma) - \rho_\infty H'(\rho_\infty) + H(\rho_\infty) \dd x \dd t
    = \int_0^\tau \int_\Omega \sigma\del_t H'(\sigma)  \dd x \dd t,
\end{align*}
we thus obtain
\begin{align*}
    \left[ \int_\Omega P(\rho | \sigma) -P(\rho | \rho_\infty) \dd x \right]_{t=0}^{t=\tau} 
    &= \left[ \int_\Omega  - \rho (H'(\sigma) - H'(\rho_\infty) ) \dd x \right]_{t=0}^{t=\tau} + \int_0^\tau \int_\Omega \sigma \del_t H'(\sigma) \dd x \dd t.
\end{align*}
We further test the continuity equation~\eqref{eq:cont.weak} with $H'(\sigma) - H'(\rho_\infty)$, which is admissible due to $\sigma=\rho_\infty$ for $|x|>R$, 
to get
\begin{align*}
    \left[ \int_\Omega \rho (H'(\sigma) - H'(\rho_\infty)) \dd x \right]_{t=0}^{t=\tau} = \int_0^\tau \int_\Omega \rho \del_t H'(\sigma) + \rho (\bu-\bomega\times x) \cdot \nabla H'(\sigma) \dd x \dd t.
\end{align*}
With these two additional identities, we find
\begin{align}\label{REI1}
\begin{split}
    &\left[ \int_\Omega \frac12 \rho |\bu - \bV|^2 + P(\rho | \sigma) \dd x \right]_{t=0}^{t=\tau} + \int_0^\tau \int_\Omega \bS(\nabla \bu) : \nabla (\bu - \bV) \dd x \dd t \\
    &= \left[ \int_\Omega \frac12 \rho |\bu - \bV|^2 + P(\rho | \rho_\infty) \dd x\right]_{t=0}^{t=\tau} + \int_0^\tau \int_\Omega \bS(\nabla \bu) : \nabla (\bu - \bV) \dd x \dd t \\
    &\qquad + \left[ \int_\Omega P(\rho | \sigma) - P(\rho | \rho_\infty) \dd x \right]_{t=0}^{t=\tau} \\
    &\leq - \int_0^\tau \int_\Omega \rho (\bu - \bomega \times x) \cdot \nabla \bU_\infty \cdot (\bu - \bU_\infty) \dd x \dd t \\
    &\qquad - \int_0^\tau \int_\Omega \rho (\bomega\times\bV)\cdot(\bu-\bV) \dd x \dd t + \int_0^\tau \int_\Omega \rho \vc f \cdot (\bu - \bV) \dd x \dd t \\
    &\qquad + \int_0^\tau \int_\Omega \rho (\bV - \bU_\infty - \bu) \cdot (\del_t \bV + ((\bu-\bomega\times x) \cdot \nabla) (\bV - \bU_\infty) ) \dd x \dd t \\
    &\qquad - \int_0^\tau \int_\Omega p(\rho) \div \bV \dd x \dd t + \left[ \int_\Omega \rho \bU_\infty \cdot (\bV - \bU_\infty) \dd x \right]_{t=0}^{t=\tau} \\
    &\qquad - \int_0^\tau \int_\Omega \rho \del_t H'(\sigma) + \rho (\bu-\bomega\times x) \cdot \nabla H'(\sigma) \dd x \dd t + \int_0^\tau \int_\Omega \sigma\del_t H'(\sigma)  \dd x \dd t.
\end{split}
\end{align}
Finally, let $\bU_\infty^R$ be a function satisfying
\begin{align*}
\bU_\infty^R \in C^\infty(\R^3),
\qquad
\div \bU_\infty^R = 0, \qquad
\bU_\infty^R(x) = \begin{cases}
\bomega \times x & \text{for }x\in B_R,\\
\vc a_\infty & \text{for }x\in\R^3 \setminus B_{2R}.
\end{cases}
\end{align*}
Existence of such a function can be shown with the help of the Bogovski\u{\i} operator, 
compare the construction of $\bV_\infty$ in~\eqref{eq:V.def} below.
We then have by Gau\ss' theorem
\begin{align*}
0 = \int_0^\tau \int_\Omega \div(p(\sigma) (\bV - \bU_\infty^R)) \dd x \dd t = \int_0^\tau \int_\Omega p(\sigma) \div \bV + (\bV - \bU_\infty^R) \cdot \nabla p(\sigma) \dd x \dd t.
\end{align*}
As $\sigma = \rho_\infty$ outside $B_R$, we have $\nabla p(\sigma) = 0$ in $\R^3 \setminus B_R$,
and we find together with $\nabla p(\sigma) = \sigma \nabla H'(\sigma)$ that
\begin{align*}
0 = \int_0^\tau \int_\Omega p(\sigma) \div \bV + \sigma (\bV - \bomega \times x) \cdot  \nabla H'(\sigma) \dd x \dd t.
\end{align*}
Moreover, using $\bU_\infty \cdot (\bV - \bU_\infty)$ as test function in the continuity equation~\eqref{eq:cont.weak} gives
\begin{align*}
    &\left[ \int_\Omega \rho \bU_\infty \cdot (\bV - \bU_\infty) \dd x \right]_{t=0}^{t=\tau} \\
    &= \int_0^\tau \int_\Omega \rho \bU_\infty \cdot \del_t \bV \dd x \dd t + \int_0^\tau \int_\Omega \rho (\bu-\bomega\times x) \cdot \nabla ( \bU_\infty \cdot (\bV - \bU_\infty) ) \dd x \dd t.
\end{align*}
Employing both identities in~\eqref{REI1} yields
\begin{align*}
     &\left[ \int_\Omega \frac12 \rho |\bu - \bV|^2 + P(\rho | \sigma) \dd x \right]_{t=0}^{t=\tau} + \int_0^\tau \int_\Omega \bS(\nabla (\bu - \bV)) : \nabla (\bu - \bV) \dd x \dd t \\
    &\leq - \int_0^\tau \int_\Omega \rho ((\bu - \bomega \times x) \cdot \nabla) \bU_\infty \cdot (\bu - \bU_\infty) \dd x \dd t \\
    &\qquad - \int_0^\tau \int_\Omega \rho (\bomega\times\bV)\cdot(\bu-\bV) \dd x \dd t + \int_0^\tau \int_\Omega \rho \vc f \cdot (\bu - \bV) \dd x \dd t \\
    &\qquad + \int_0^\tau \int_\Omega \rho (\bV - \bu) \cdot ( \del_t \bU  + ((\bu-\bomega\times x) \cdot \nabla) (\bV - \bU_\infty) ) \dd x\dd t \\
    &\qquad + \int_0^\tau \int_\Omega \rho ((\bu-\bomega\times x) \cdot \nabla) \bU_\infty \cdot (\bV - \bU_\infty) \dd x \dd t \\
    &\qquad - \int_0^\tau \int_\Omega \bS(\nabla \bV) : \nabla (\bu - \bV) \dd x \dd t \\
    &\qquad - \int_0^\tau \int_\Omega ( p(\rho) - p(\sigma)) \div \bV \dd x \dd t \\
    &\qquad - \int_0^\tau \int_\Omega (\rho - \sigma) \del_t H'(\sigma) + (\rho (\bu-\bomega\times x) - \sigma (\bV - \bomega \times x)) \cdot \nabla H'(\sigma) \dd x \dd t.
\end{align*}
The first and the fourth integral on the right-hand side 
add up to
\begin{align*}
    &- \int_0^\tau \int_\Omega \rho ((\bu - \bomega \times x) \cdot \nabla) \bU_\infty \cdot (\bu - \bU_\infty) \dd x \dd t + \int_0^\tau \int_\Omega \rho ((\bu-\bomega\times x) \cdot \nabla) \bU_\infty \cdot (\bV - \bU_\infty) \dd x \dd t \\
    &= \int_0^\tau \int_\Omega \rho ((\bu-\bomega\times x) \cdot \nabla) \bU_\infty \cdot (\bV - \bu) \dd x \dd t, 
\end{align*}
so that we finally find~\eqref{REI}, provided that $(\sigma, \bU) = (\rho_\infty, \vc a_\infty)$ for $|x|>R$.
\end{proof}

To omit the additional assumptions from Lemma~\ref{lem:rei.compact},
we proceed with an approximation argument.
\begin{proof}[Proof of Theorem~\ref{thm:rei}]
    Let $\phi\in C^\infty(\R)$ with $\phi(s)=1$ for $s<1$ and $\phi(s)=0$ for $s>2$, and define $\phi_R(x)=\phi(|x|/R)$ for $R>1$.
    We approximate $(\sigma, \bV)$ by pairs $(\sigma_R,\bV_R)$ defined as
    \[
    \sigma_R(t,x)= \phi_R(x) \sigma(t,x) + (1-\phi_R(x)) \rho_\infty,
    \qquad
    \bV_R(t,x)=\phi_R(x) \bV(t,x) + (1-\phi_R(x)) \ba_\infty
    \]
    such that $\supp(\sigma_R-\rho_\infty, \bV_R-\ba_\infty) \subset (0,T)\times B_{2 R}$. By Lemma~\ref{lem:rei.compact}, we then have the REI
    \begin{equation}\label{REI.R}
    \begin{aligned}
     &\left[ \int_\Omega E(\rho, \bu | \sigma_R, \bV_R)  \dd x \right]_{t=0}^{t=\tau} + \int_0^\tau \int_\Omega \bS(\nabla \bu ) : \nabla (\bu - \bV_R) \dd x \dd t \\
    &\leq 
    \int_0^\tau \int_\Omega \rho (\bV_R - \bu) \cdot ( \del_t \bV_R  + \bomega\times\bV_R  + ((\bu-\bomega\times x) \cdot \nabla) \bV_R  - \vc f ) \dd x\dd t \\
    &\qquad - \int_0^\tau \int_\Omega ( p(\rho) - p(\sigma_R)) \div \bV_R \dd x \dd t \\
    &\qquad - \int_0^\tau \int_\Omega (\rho - \sigma_R) \del_t H'(\sigma_R) + (\rho (\bu-\bomega\times x) - \sigma_R (\bV_R - \bomega \times x)) \cdot \nabla H'(\sigma_R) \dd x \dd t
\end{aligned}
\end{equation}
for a.a.~$\tau\in(0,T)$.
To shorten the presentation, we have collected the viscous terms on the left-hand side.
For conclusion of~\eqref{enIneq}, we now pass to the limit $R\to\infty$ 
in each term separately.

First of all, we have
\[
\begin{aligned}
&E(\rho, \bu | \sigma, \bV) - E(\rho, \bu | \sigma_R, \bV_R)
\\
&=\frac{1}{2}\rho\big[-2(\bu-\ba_\infty)\cdot(\bV-\ba_\infty)
+|\bV-\ba_\infty|^2
+2(\bu-\ba_\infty)\cdot(\bV_R-\ba_\infty)
-|\bV_R-\ba_\infty|^2\big]
\\
&\qquad
+\int_0^1 \frac{\mathrm d}{\mathrm d\theta}P(\rho\mid (1-\theta)\sigma_R+\theta \sigma)\dd\theta
\\
&=\frac{1}{2}\rho\big[2(\phi_R-1)(\bu-\ba_\infty)\cdot(\bV-\ba_\infty)
+(1-\phi_R^2)|\bV-\ba_\infty|^2\big]
\\
&\qquad
-(1-\phi_R)(\sigma -\rho_\infty)\int_0^1 H''\big(\sigma - (1-\theta)(1-\phi_R)(\sigma-\rho_\infty)\big)
\big[\rho-\sigma - (1-\theta)(1-\phi_R)(\sigma-\rho_\infty) \big]\dd\theta,
\end{aligned}
\]
where we used~\eqref{eq:H.def}.
Since $\sqrt{\rho}|\bu-\ba_\infty|\in L^\infty(0,T;L^2(\Omega))$ 
and $\rho = (\rho - \rho_\infty) + \rho_\infty \in L^\infty(0,T; [L^2 + L^\gamma + L^\infty] (\Omega))$, see Definition~\ref{def:wksol},
the integrability properties of $\bV-\ba_\infty$ stated in~\eqref{regStr} imply
that the first term on the right-hand converges to $0$ in $L^\infty(0,T;L^1(\Omega))$ as $R\to\infty$, which is due to the fact that $\phi_R, \phi_R^2 \stackrel{*}{\rightharpoonup} 1$ in $L^\infty(\Omega)$.
For the second term, we use the identity
\[
\begin{aligned}
&(1-\phi_R)(\sigma-\rho_\infty)\big[\rho-\sigma - (1-\theta)(1-\phi_R)(\sigma-\rho_\infty) \big]
\\
&\quad
=(1-\phi_R)(\sigma-\rho_\infty)(\rho-\rho_\infty)-(1-\phi_R)\big[1- (1-\theta)(1-\phi_R)\big](\sigma-\rho_\infty)^2 
\end{aligned}
\]
and conclude the convergence to $0$ in a similar way,
using that convex combinations of $\sigma$ and $\rho_\infty$ are bounded from above and below.
In summary, we see $ E(\rho, \bu | \sigma_R, \bV_R)\to E(\rho, \bu | \sigma, \bV)$ in $L^\infty(0,T;L^1(\Omega))$ as $R\to\infty$.

For the diffusive term, note that
\[
\begin{aligned}
&\bS(\nabla \bu) : \nabla (\bu - \bV)-\bS(\nabla \bu) : \nabla (\bu - \bV_R)
=\bS(\nabla\bu): \big[
(1-\phi_R)\nabla\bV - (\bV-\ba_\infty)\otimes\nabla\phi_R)\big].
\end{aligned}
\]
Due to the convergences 
$\phi_R\stackrel{*}{\rightharpoonup} 1$ in $L^\infty(\Omega)$ and 
$\nabla\phi_R\rightharpoonup 0$ in $L^3(\Omega)$ as $R\to\infty$,
which one readily verifies,
and due to~\eqref{regStr}, 
which also implies $\bV-\ba_\infty\in L^2(0,T;L^6(\Omega))$
by Sobolev inequality,
we conclude the convergence
$\bS(\nabla \bu) : \nabla (\bu - \bV_R)\to\bS(\nabla \bu) : \nabla (\bu - \bV)$
in $L^1((0,T)\times\Omega)$ as $R\to\infty$.

To derive convergence of the first term on the right-hand side of~\eqref{REI.R}, 
we observe that
\[
\begin{aligned}
&\rho (\bV - \bu) \cdot ( \del_t \bV  + \bomega\times\bV  + ((\bu-\bomega\times x) \cdot \nabla) \bV  - \vc f )
\\
&\qquad
-\rho (\bV_R - \bu) \cdot ( \del_t \bV_R  + \bomega\times\bV_R  + ((\bu-\bomega\times x) \cdot \nabla) \bV_R  - \vc f )
\\
&=(1-\phi_R)\rho(\bV-\ba_\infty)\cdot ( \del_t \bV  + \bomega\times\bV  + ((\bu-\bomega\times x) \cdot \nabla) \bV  - \vc f )
\\
&\qquad
+ (1-\phi_R)\rho (\bV - \bu) \cdot (\del_t \bV  + \bomega\times(\bV-a_\infty)  + ((\bu-\bomega\times x) \cdot \nabla) \bV) 
\\
&\qquad+\rho (\bV - \bu)\cdot (\bV-\ba_\infty) ((\bu-\bomega\times x) \cdot \nabla\phi_R).
\end{aligned}
\]
In virtue of the integrability assumptions from~\eqref{eq:reg.weak} and~\eqref{regStr}
and the aforementioned convergence properties of $\phi_R$ and $\nabla\phi_R$,
one concludes that all terms on the right-hand side converge to $0$ in $L^1((0,T)\times\Omega)$ as $R\to\infty$.

For the remaining terms in~\eqref{REI.R}, we can use similar arguments as before.
In particular, for the last term, we use the representation
\[
\begin{aligned}
&(\rho - \sigma) \del_t H'(\sigma) + (\rho (\bu-\bomega\times x) - \sigma (\bV - \bomega \times x)) \cdot \nabla H'(\sigma)
\\
&
=H''(\sigma)\big[(\rho{-}\rho_\infty)-(\sigma{-}\rho_\infty)\big]\del_t \sigma
+H''(\sigma)\big[\sqrt{\rho} \sqrt{\rho}(\bu-\ba_\infty)
-\sigma (\bV-\ba_\infty)
+ (\rho{-}\sigma)(\ba_\infty-\bomega\times x)\big]
\cdot\nabla \sigma.
\end{aligned}
\]
Since $\rho-\rho_\infty\in L^\infty(0,T; [L^2 + L^\gamma](\Omega))$,
and $\sqrt{\rho}(\bu-\ba_\infty)\in L^\infty(0,T;L^2(\Omega))$,
this explains the regularity assumptions on $\del_t \sigma$ and $\nabla \sigma$ in~\eqref{regStr}.
Finally, passing to the limit $R\to\infty$ in~\eqref{REI.R} leads to~\eqref{REI} 
and completes the proof.
\end{proof}

\begin{Remark}
    In the proof of Lemma~\ref{lem:rei.compact} 
    we did not make use of any structural assumptions on the pressure $p$,
    so that the statement keeps valid for any $p\in C^0([0,\infty)) \cap C^1((0,\infty))$.
    For the approximation argument to 
    obtain Theorem~\ref{thm:rei},
    we use the growth conditions from~\eqref{assPress},
    but the restriction on $\gamma$ could be relaxed to $\gamma>1$.
\end{Remark}

\section{Weak-strong uniqueness}
\label{sec:wsu}

In this section, we show that a weak solution to \eqref{NSE} coincides with a strong one emanating from the same initial data, at least as long as the latter exists. 
To this end, we make use of the following Korn-type inequality.

\begin{Lemma}\label{lem:Korn}
    Let $p\in[1,\infty)$. For any $\bw\in L^p(\R^3;\R^3)$ with $\nabla \bw\in L^2(\R^3;\R^{3\times 3})$ it holds
    \begin{align*}
        \sqrt{2} \|\nabla \bw\|_{L^2(\R^3)} \leq  \Big\| \nabla \bw + \nabla^T \bw - \frac23 \div \bw \Id \Big\|_{L^2(\R^3)}.
    \end{align*}
\end{Lemma}
\begin{proof}
    First, let $\bw\in C_c^\infty(\R^3;\R^3)$. 
    Then integration by parts yields
    \[
    \Big\| \nabla \bw + \nabla^T \bw - \frac23 \div \bw \Id \Big\|_{L^2(\R^3)}^2
    =\int_{\R^3} 2|\nabla \bw|^2 + 2\nabla \bw:\nabla\bw^T - \frac{4}{3}(\div\bw)^2 \dd x
    =\int_{\R^3} 2|\nabla \bw|^2 +\frac{2}{3}(\div\bw)^2 \dd x.
    \]
    Since $(\div\bw)^2\geq 0$, this gives the asserted inequality for $\bw\in C_c^\infty(\R^3;\R^3)$.
    The general case follows by a standard approximation argument.
\end{proof}

Using the REI established in Theorem~\ref{thm:rei},
we derive the following weak-strong uniqueness result.
Our argument follows the presentation in \cite{Basaric2022} and \cite{FeireislNovotny2021weak}.

\begin{Theorem}
    \label{thm:wsu}
    Let $(\rho, \bu)$ be a finite-energy weak solution to system \eqref{NSE} with initial data $(\rho_0, \bu_0)$.
    Let further $(\sigma, \bV) = (r, \bv)$ belong to the regularity class specified in \eqref{regStr} and satisfy 
    \begin{equation}
    \label{eq:wsu.regularity}
        \div \bS(\nabla \bv)\in L^\infty(0,T; [L^3 \cap L^{\frac{6\gamma}{5\gamma-6}}] (\Omega)),
    \qquad
    \nabla\bv\in L^\infty(0,T;L^\infty(\Omega)).
    \end{equation}
    Moreover, assume that $(r, \bv)$ satisfies the system~\eqref{NSE} with the same initial data as $(\rho, \bu)$.
    Then $\rho = r$ and $\bu = \bv$ a.e.~in $(0,T) \times \Omega$.
\end{Theorem}

\begin{proof}
We use the momentum equation satisfied by $(r, \bv)$, specifically,
\begin{align*}
    \del_t \bv + \bomega\times\bv + ((\bu -\bomega\times x) \cdot \nabla) \bv - \vc f = \frac{1}{r} \div \bS(\nabla \bv) + ((\bu - \bv) \cdot \nabla) \bv - \frac{1}{r} \nabla p(r).
\end{align*}
Moreover, the continuity equation for $(r, \bv)$ implies
\begin{align*}
    (r-\rho) (\del_t H'(r) + \nabla H'(r) \cdot (\bv - \bomega \times x)) &= (r-\rho)H''(r) (\del_t r + \nabla r \cdot (\bv - \bomega \times x)) \\
    &= (r-\rho) H''(r) (-r\div (\bv - \bomega \times x)) \\
    &= (\rho - r)  p'(r)\div \bv.
\end{align*}
Then the REI~\eqref{REI}, together with $r^{-1} \nabla p(r) = \nabla H'(r)$, yields
\begin{equation}\label{mezi1}
\begin{aligned}
    &\left[ \int_\Omega E(\rho, \bu | r, \bv) \dd x \right]_{t=0}^{t=\tau} + \int_0^\tau \int_\Omega \bS(\nabla (\bu - \bv)) : \nabla (\bu - \bv) \dd x \dd t 
    \\
    &\leq \int_0^\tau \int_\Omega \rho (\bv - \bu) \cdot \Bigg[ \frac{1}{r}\div \bS(\nabla \bv)+((\bu - \bv) \cdot \nabla) \bv -\frac{1}{r}\nabla p(r)\Bigg]\dd x \dd t \\
    &\quad - \int_0^\tau \int_\Omega \bS(\nabla \bv):\nabla(\bu-\bv) \dd x \dd t 
    - \int_0^\tau \int_\Omega (p(\rho)-p(r))\div\bv \dd x \dd t\\
    &\quad - \int_0^\tau \int_\Omega (\rho-r)p'(r)\div\bv-\rho(\bu-\bv)\cdot\nabla H'(r)\dd x \dd t
    \\
    &= \int_0^\tau \int_\Omega \rho (\bv - \bu) \cdot ((\bu - \bv) \cdot \nabla) \bv \dd x \dd t + \int_0^\tau \int_\Omega \Big( \frac{\rho}{r} - 1 \Big) (\bv - \bu) \cdot \div \bS(\nabla \bv) \dd x \dd t \\
    &\quad - \int_0^\tau \int_\Omega \big(p(\rho) - p'(r) (\rho - r) - p(r)\big)\div \bv  \dd x \dd t,
\end{aligned}
\end{equation}
where we integrated by parts in the second term of the first inequality and used that $(\bu - \bv) |_{\del \Omega} = 0$. To treat the viscous terms on the left-hand side, we observe that
\begin{align*}
    \int_\Omega \bS(\nabla(\bu - \bv)) : \nabla(\bu - \bv) \dd x &\geq \frac{\mu}{2} \Big\| \nabla (\bu - \bv) + \nabla^T(\bu - \bv) - \frac23 \div(\bu - \bv) \Id \Big\|_{L^2(\Omega)}^2 \\
    &\geq \mu \|\nabla(\bu - \bv)\|_{L^2(\Omega)}^2,
\end{align*}
where we prolonged $\bu$ and $\bv$ by $\bomega \times x$ outside of $\Omega$, and used the Korn-type inequality from Lemma~\ref{lem:Korn}.
The first integral on the right-hand side of \eqref{mezi1} is easily estimated as
\begin{align*}
    \int_0^\tau \int_\Omega \rho (\bv - \bu) \cdot((\bu - \bv) \cdot \nabla) \bv \dd x \dd t &\leq \|\nabla \bv\|_{L^\infty(0,\tau; L^\infty(\Omega))} \int_0^\tau \int_\Omega  \rho |\bu - \bv|^2 \dd x \dd t \\
    &\leq 2\|\nabla \bv\|_{L^\infty(0,\tau; L^\infty(\Omega))} \int_0^\tau \int_\Omega E(\rho, \bu | r, \bv) \dd x \dd t.
\end{align*}
Introducing the essential and residual part of a function $f$ as
\begin{align*}
    [f]_{\ess} = \chi_{\frac12 \leq \rho/r \leq 2} f, && [f]_{\res} = f - [f]_{\ess},
\end{align*}
we find that the relative internal energy is coercive with respect to $\rho$ 
in the sense that
\begin{equation}\label{eq:coerciveP}
    [\rho]_{\res}^\gamma + [1]_{\res} + [\rho - r]_{\ess}^2 \leq C P(\rho | r),
\end{equation}
where $C>0$ only depends on $\underline r$ and $\overline r$ if $\underline r\leq r\leq \overline r$,
see~\cite[Equ.~(4.23)]{FeireslNovotny2022openfluidsystems} or~\cite[Equ.~(5.9)]{JessleJinNovotny2013}.
For the second integral on the right-hand side of \eqref{mezi1}, we now calculate for the residual part that
\begin{align*}
    &\int_\Omega \Big[ \frac{\rho}{r} - 1 \Big]_{\res} (\bv - \bu) \cdot \div \bS(\nabla \bv) \dd x 
    \\
    &\quad\leq  \big(\|[\rho/r]_{\res}\|_{L^\gamma(\Omega)} + \| [1]_{\res} \|_{L^\gamma(\Omega)} \big)\|\bv - \bu\|_{L^6(\Omega)} \|\div \bS(\nabla \bv)\|_{L^\frac{6\gamma}{5\gamma-6}(\Omega)} \\
    &\quad\leq C \|\div \bS(\nabla \bv)\|_{L^\frac{6\gamma}{5\gamma-6}(\Omega)} ^2 (\underline r^{-2}\| [\rho]_{\res} \|_{L^\gamma(\Omega)}^2 + \| [1]_{\res} \|_{L^\gamma(\Omega)}^2) + \delta \|\nabla(\bv - \bu)\|_{L^2(\Omega)}^2,
\end{align*}
where $\delta>0$ is arbitrary, and where we used the Sobolev inequality.
In view of the coercivity estimate~\eqref{eq:coerciveP},
we find for $\frac{6}{5}\leq\gamma \leq 2$ that
\begin{align*}
    \|[\rho]_{\res}\|_{L^\gamma(\Omega)}^2 + \|[1]_{\res}\|_{L^\gamma(\Omega)}^2 &\leq  C \Big( \int_\Omega E(\rho, \bu | r, \bv) \dd x \Big)^\frac{2}{\gamma} 
    \leq  C \int_\Omega E(\rho, \bu | r, \bv) \dd x,
\end{align*}
since also $E(\rho, \bu | r, \bv) \in L^\infty(0,T; L^1(\Omega))$.
If $\gamma > 2$ and $\div \bS(\nabla \bv) \in L^2(0,T; L^3(\Omega))$, we may estimate
\begin{align*}
    &\int_\Omega \Big[ \frac{\rho}{r} - 1 \Big]_{\res} (\bv - \bu) \cdot \div \bS(\nabla \bv) \dd x 
    \\
    &\quad
    \leq \big(\|[\rho/r]_{\res}\|_{L^2(\Omega)} + \| [1]_{\res} \|_{L^2(\Omega)} \big)
    \|\bv - \bu\|_{L^6(\Omega)}  \|\div \bS(\nabla \bv)\|_{L^3(\Omega)} \\
    &\quad
    \leq C \|\div \bS(\nabla \bv)\|_{L^3(\Omega)}^2 (\underline{r}^{-2}\| [\rho]_{\res} \|_{L^2(\Omega)}^2 + \| [1]_{\res} \|_{L^2(\Omega)}^2) + \delta \|\nabla(\bv - \bu)\|_{L^2(\Omega)}^2
\end{align*}
by the Sobolev inequality.
By $\gamma > 2$ and Young's inequality, we have
\begin{align*}
    [\rho]_{\res}^2 + [1]_{\res}^2\lesssim [\rho]_{\res}^\gamma + [1]_{\res} \lesssim E(\rho, \bu | r, \bv)
\end{align*}
due to~\eqref{eq:coerciveP}.
Consequently, for any $\gamma \geq \frac65$, we arrive at
\begin{align}
\label{eq:reldiff.res}
    \int_\Omega \Big[ \frac{\rho}{r} - 1 \Big]_{\res} (\bv - \bu) \cdot \div \bS(\nabla \bv) \dd x 
    \leq \|\div \bS(\nabla \bv)\|_{L^p(\Omega)} ^2
    \int_\Omega E(\rho, \bu | r, \bv) \dd x + \delta \|\nabla(\bv - \bu)\|_{L^2(\Omega)}^2
\end{align}
with $p=\max\{\frac{6\gamma}{5\gamma-6},3\}$.
An analogous argument for the essential part leads to
\begin{equation}
\label{eq:reldiff.ess}
\begin{aligned}
    &\int_\Omega \Big[ \frac{\rho}{r} - 1 \Big]_{\ess} (\bv - \bu) \cdot \div \bS(\nabla \bv) \dd x \leq \underline{r}^{-1} \|[\rho - r]_{\ess}\|_{L^2(\Omega)} \|\bu - \bv\|_{L^6(\Omega)} \|\div \bS(\nabla \bv)\|_{L^3(\Omega)}\\
    &\qquad\leq C \|\div \bS(\nabla \bv)\|_{L^3(\Omega)}^2 \int_\Omega E(\rho, \bu | r, \bv) \dd x + \delta \|\nabla(\bu - \bv)\|_{L^2(\Omega)}^2.
\end{aligned}
\end{equation}
The very last pressure integral in \eqref{mezi1} is handled similarly. 
Here we use Taylor's theorem for the essential part, 
exploiting $p\in C^2((0,\infty))$,
and Young's inequality for the residual part, and find (compare \cite[Section~5.4, Step~4]{JessleJinNovotny2013})
\begin{align*}
    &\int_\Omega \div \bv \big(p(\rho) - p'(r)(\rho - r) - p(r)\big) \dd x 
    \\
    &\quad
    \leq C \|\div \bv \|_{L^\infty(\Omega)} \big(\|[\rho - r]_{\ess}\|_{L^2(\Omega)}^2 + \|[p(\rho)]_{\res}\|_{L^1(\Omega)} + \|[1]_{\res}\|_{L^1(\Omega)}\big) \\
    &\quad
    \leq C \|\div \bv \|_{L^\infty(\Omega)} \int_\Omega E(\rho, \bu | r, \bv) \dd x,
\end{align*}
where we used $p(\rho) \lesssim 1 + \rho^\gamma$ and~\eqref{eq:coerciveP}.

Gathering all estimates, choosing $\delta>0$ small enough to absorb the dissipative terms to the left-hand side of~\eqref{mezi1},
and employing~\eqref{eq:wsu.regularity}, 
we are left with the inequality
\begin{align*}
\left[ \int_\Omega E(\rho, \bu | r, \bv) \dd x \right]_{t=0}^{t=\tau} + \|\nabla( \bu - \bv) \|_{L^2((0,T) \times \Omega)}^2 \leq C \int_0^\tau \int_\Omega E(\rho, \bu | r, \bv) \dd x \dd t.
\end{align*}
Applying Gr\"onwall's inequality completes the proof.
\end{proof}

\section{Low Mach number limit}
\label{sec:lowMach}

In this section, we consider system \eqref{NSE}, but with a rescaled pressure of the form
\begin{equation}\label{NSE-LM}
\left\{\begin{aligned}
    \del_t \rho + \div(\rho (\bu - \bomega \times x)) &= 0 && \mbox{in } (0,T)\times \Omega,\\
\del_t (\rho\bu) + \div ( \rho (\bu - \bomega \times x) \otimes \bu ) + \rho \bomega \times \bu &=  \rho \vc f +  \div \bS(\nabla \bu) - \frac{1}{\eps^2} \nabla p(\rho) && \mbox{in }(0,T)\times \Omega,\\
\bu &=  \bomega \times x && \mbox{on } (0,T)\times \del \Omega,\\
\rho(\cdot,x) \to \rho_\infty,\ \bu(\cdot,x) &\to \vc a_\infty && \mbox{as } |x| \to \infty,\\
\rho(0,\cdot)=\rho_0, \ (\rho \bu)(0,\cdot) &= \vc \rho_0\bu_0 && \mbox{in } \Omega.
\end{aligned}\right.
\end{equation}
The additional factor $\eps^{-2}$ in front of the pressure plays the role of a low Mach number. In turn, we obtain the rescaled relative energy
\begin{align*}
    E_\eps(\rho, \bu | \sigma, \bV) = \frac12 \rho |\bu - \bV|^2 + \frac{1}{\eps^2} P(\rho | \sigma)
\end{align*}
and find the relative energy inequality of the form
\begin{equation}\label{REI-LM}
\begin{aligned}
     &\left[ \int_\Omega E_\eps(\rho, \bu | \sigma, \bV) \dd x \right]_{t=0}^{t=\tau} + \int_0^\tau \int_\Omega \bS(\nabla (\bu - \bV)) : \nabla (\bu - \bV) \dd x \dd t \\
    &\leq \int_0^\tau \int_\Omega \rho (\bV - \bu) \cdot ( \del_t \bV  + \bomega\times\bV  + ((\bu-\bomega\times x) \cdot \nabla) \bV  - \vc f ) \dd x \dd t \\
    &\quad - \int_0^\tau \int_\Omega \bS(\nabla \bV) : \nabla (\bu - \bV) \dd x \dd t \\
    &\quad - \frac{1}{\eps^2} \int_0^\tau \int_\Omega ( p(\rho) - p(\sigma)) \div \bV \dd x \dd t \\
    &\quad - \frac{1}{\eps^2} \int_0^\tau \int_\Omega (\rho - \sigma) \del_t H'(\sigma) + (\rho (\bu-\bomega\times x) - \sigma (\bV - \bomega \times x)) \cdot \nabla H'(\sigma) \dd x \dd t,
\end{aligned}
\end{equation}
which follows from Theorem~\ref{thm:rei}.
We shall use~\eqref{REI-LM} to quantify the convergence 
of solutions to~\eqref{NSE-LM} (for $\varepsilon>0$) as $\varepsilon\to 0$.
This limit passage formally leads to the incompressible Navier--Stokes equations~\eqref{NSE-LM-Limit}.
We specify this low Mach number limit, where the convergence is stated in terms of the relative energy.

\begin{Theorem}\label{thm:singLim}
    For $\eps \in (0, 1)$ let $(\rho_\eps, \bu_\eps)$ be a finite energy weak solution to system \eqref{NSE-LM} in the sense of Definition~\ref{def:wksol}, emanating from the initial data $(\rho_{\eps, 0}, \bu_{\eps, 0})$. Let $(\bU,\Pi)$ be such that the pair $(\sigma,\bV)=(\rho_\infty, \bU)$ lies in the regularity class specified in \eqref{regStr}, with the additional properties
    \begin{align}\label{eq:singLim.assstrong}
    \begin{aligned}
        \div \bS(\nabla \bU) &\in L^2(0,T; [L^2\cap L^3 \cap L^\frac{6\gamma}{5\gamma-6}] (\Omega)), 
        & \partial_t \bU &\in L^\infty(0,T; L^2(\Omega;\R^3)),
        \\
        \nabla \Pi &\in L^2(0,T; [L^2 \cap L^3 \cap L^\frac{6\gamma}{5\gamma-6}] (\Omega)), & 
        \Pi &\in L^2((0,T) \times \Omega),
    \end{aligned}
    \end{align}
    and such that system \eqref{NSE-LM-Limit} is satisfied
    for some initial data $\bU_0$ with $\bU_0-\ba_\infty\in L^2(\Omega;\R^3)$.
    If
    \begin{align}\label{ass:initDat}
        \lim_{\eps \to 0} \int_\Omega E_\eps(\rho_{\eps, 0}, \bu_{\eps, 0} | \rho_\infty, \bU_0) \dd x = 0,
    \end{align}
    then
    \begin{equation}
    \label{eq:limit.lm}
        \lim\limits_{\eps \to 0} \sup_{\tau \in [0, T]} \int_\Omega E_\eps(\rho_{\eps}, \bu_{\eps} | \rho_\infty, \bU)(\tau) \dd x = 0.
    \end{equation}
\end{Theorem}

The existence of local-in-time strong solutions $(\bU,\Pi)$ to \eqref{NSE-LM-Limit} 
was shown in~\cite[Therorem~1]{GaldiSilvestre_NSErot2005} for $\ba_\infty=0$, and global-in-time strong solutions were found in \cite[Theorem~2]{GaldiSilvestre_NSErot2005} for small initial data and small rotation. We emphasize that the regularity of the strong solutions in \cite{GaldiSilvestre_NSErot2005} differs from ours.
For extensions of this existence result, 
allowing even for time-dependence of $\ba_\infty$ and $\bomega$, 
see the monograph~\cite{NKNP}, 
the recent article~\cite{AsamiHishida2025reggenOseen}, 
and references therein.

Under additional assumptions on the pressure $\Pi$, we can show the following quantified version of Theorem~\ref{thm:singLim}:
\begin{Theorem}\label{thm:singLimRate}
    Under the assumptions of Theorem~\ref{thm:singLim}, assume additionally that
    \begin{align}
    \label{eq:sigLimRate.pressurereg}
\begin{aligned}
\partial_t\Pi,\,
(\ba_\infty-\bomega\times x)\cdot \nabla\Pi
 \in L^1(0,T; [L^1 \cap L^2 \cap L^{\frac{\gamma}{\gamma-1}}] (\Omega)),
&&
\Pi(0,\cdot) \in [L^1 \cap L^2 \cap L^{\frac{\gamma}{\gamma-1}}] (\Omega).
\end{aligned}
\end{align}
Then
\begin{align*}
    &\sup_{\tau \in [0,T]} \int_\Omega E_\eps (\rho_\eps, \bu_\eps | \rho_\infty, \bU)(\tau) \dd x + \|\nabla(\bu_\eps - \bU)\|_{L^2(0,T; L^2(\Omega))}^2 \\
    &\leq C \int_\Omega E_\eps (\rho_{\eps, 0}, \bu_{\eps, 0} | \rho_\infty, \bU_0) \dd x + C(\eps + \eps^\frac2\gamma),
\end{align*}
where the constant $C>0$ is independent of $\eps$.
\end{Theorem}

\begin{Remark}
    Thanks to the coercivity properties of the relative energy, we also have
    \begin{align*}
        &\|[\bu_\eps - \bU]_\esseps\|_{L^\infty(0,T; L^2(\Omega))}^2 + \frac{1}{\eps^2} \|[\rho_\eps - \rho_\infty]_{\ess, \eps}\|_{L^\infty(0,T; L^2(\Omega))}^2 + \frac{1}{\eps^2} \|[\rho_\eps]_{\res, \eps}\|_{L^\infty(0,T; L^\gamma(\Omega))}^\gamma \\
        &\leq C \int_\Omega E_\eps (\rho_{\eps, 0}, \bu_{\eps, 0} | \rho_\infty, \bU_0) \dd x + C(\eps + \eps^\frac2\gamma),
    \end{align*}
    where we define the $\varepsilon$-dependent essential and residual part of a function $f$ as
    \begin{equation}
    \label{eq:ess.res.eps}
    [f]_{\esseps} = \chi_{\frac12 \leq \rho_\eps / \rho_\infty \leq 2} f, 
    \qquad
    [f]_{\reseps} = f - [f]_{\esseps}.
    \end{equation}
\end{Remark}

The following two subsections are devoted to the proofs of Theorems~\ref{thm:singLim} and \ref{thm:singLimRate}.

\subsection{Qualitative convergence, proof of Theorem~\ref{thm:singLim}}
We split the proof of Theorem~\ref{thm:singLim} 
into several lemmas.
First, we use the REI~\eqref{REI-LM} to derive suitable bounds on $(\rho_\varepsilon, \bu_\varepsilon)$.

\begin{Lemma}\label{lem:bounds}
    The weak solutions $(\rho_\eps,\bu_\eps)$ satisfy the $\eps$-uniform bounds
    \begin{align*}
    \|[\rho_\eps - \rho_\infty]_{\esseps}\|_{L^\infty(0,T; L^2(\Omega))}^2 + \|[\rho_\eps]_{\reseps}\|_{L^\infty(0,T; L^\gamma(\Omega))}^\gamma + \|[1]_{\reseps}\|_{L^\infty(0,T; L^1(\Omega))} &\lesssim \eps^2, \\
    \|[\bu_\eps - \bU_\infty]_{\esseps}\|_{L^\infty(0,T;L^2(\Omega))}^2 + \|\nabla(\bu_\eps - \bU_\infty)\|_{L^2(0,T; L^2(\Omega))}^2 &\lesssim 1,
\end{align*}
where the essential and residual parts are defined as in~\eqref{eq:ess.res.eps}.
\end{Lemma}

\begin{proof}
\providecommand{\Bog}{\mathcal{B}}
Since the limiting system \eqref{NSE-LM-Limit} is incompressible, we wish to take an incompressible test function that is not ``too far away'' from $\bu_\eps$ in the REI \eqref{REI}. To this end, we first recall the notion of the Bogovski\u{\i} operator. 
Let $q\in(1,\infty)$.
By $L_0^q(B_1 \setminus \ms)$ we denote the set of all functions $f \in L^q(B_1 \setminus \ms)$ with $\int_{B_1 \setminus \ms} f \dd x = 0$.
Since $B_1 \setminus \ms$ is a bounded Lipschitz domain, there exists a bounded linear operator $\Bog: (L_0^q \cap C^\infty)(B_1 \setminus \ms) \to C_0^\infty(B_1 \setminus \ms; \R^3)$ such that
\begin{align*}
    \div \Bog(f) &= f \ \text{ in } \ B_1 \setminus \ms, \\
    \|\Bog(f)\|_{W^{1,q}(B_1 \setminus \ms)} &\leq C \|f\|_{L^q(B_1 \setminus \ms)}
\end{align*}
for any $f \in (L_0^q \cap C^\infty)(B_1 \setminus \ms)$, 
see e.g. \cite[Theorem III.3.1]{Galdi_NSE2011}. 
Recalling the definition of $\bU_\infty$ from \eqref{Uinfty}, we immediately see that
\begin{align*}
    \int_{B_1} \div \bU_\infty \dd x &= \int_{\del B_1} \bU_\infty \cdot \vc n \dd S = \int_{\del B_1} \vc a_\infty \cdot \vc n \dd S = 0, \\
    \int_{\ms} \div \bU_\infty \dd x &= \int_{\ms} \div(\bomega \times x) \dd x = 0,
\end{align*}
and hence $\div \bU_\infty \in (L_0^q \cap C^\infty)(B_1 \setminus \ms)$ for any $q\in(1,\infty)$. In turn, we define
\begin{equation}
\label{eq:V.def}
    \bV_\infty = \bU_\infty - \Bog (\div \bU_\infty) \in C^\infty(\R^3; \R^3).
\end{equation}
Note especially that we have $\bV_\infty |_{\partial \ms} = \bU_\infty |_{\partial \ms} = \bomega \times x$ and ${\rm supp} (\bV_\infty - \vc a_\infty) \subset \overline B_1$
due to $\Bog(f)|_{\partial(B_1 \setminus \ms)} = 0$ for any $f \in L_0^q(B_1 \setminus \ms)$. Moreover, $\div \bV_\infty = 0$ in $\R^3$ and
\begin{align}\label{estV}
    \|\bV_\infty - \vc a_\infty\|_{W^{1,q}(\R^3)} \leq C \Big( \|\bU_\infty - \vc a_\infty\|_{W^{1,q}(B_1)} + \|\div \bU_\infty \|_{L^q(B_1)} \Big) \leq C \|\bU_\infty - \vc a_\infty\|_{W^{1,q}(B_1)}.
\end{align}
By~\eqref{ass:initDat}, we further have
\begin{align*}
\int_\Omega E_\eps(\rho_{\eps, 0}, \bu_{\eps, 0} | \rho_\infty, \bU_0) \dd x \lesssim 1 .
\end{align*}
In particular, since $\bU_0-\ba_\infty\in L^2(\Omega)$, 
this implies
\begin{align*}
\int_\Omega E_\eps(\rho_{\eps, 0}, \bu_{\eps, 0} | \rho_\infty, \bV_\infty) \dd x \lesssim 1.
\end{align*}

Now we use the REI \eqref{REI-LM} with test functions $(\rho_\infty, \bV_\infty)$.
In view of $\del_t H'(\rho_\infty) = 0$ and $\nabla H'(\rho_\infty) = 0$, we find 
\begin{equation}\label{REI2}
\begin{aligned}
     &\left[ \int_\Omega E_\eps(\rho_\eps, \bu_\eps | \rho_\infty, \bV_\infty) \dd x\right]_{t=0}^{t=\tau} + \int_0^\tau \int_\Omega \bS(\nabla (\bu_\eps - \bV_\infty)) : \nabla (\bu_\eps - \bV_\infty) \dd x \dd t \\
    &\leq \int_0^\tau \int_\Omega \rho_\eps (\bV_\infty - \bu_\eps) \cdot (\bomega \times \bV_\infty  + ((\bV_\infty - \bomega\times x) \cdot \nabla) \bV_\infty - \vc f ) \dd x \dd t \\
    &\quad + \int_0^\tau \int_\Omega \rho_\eps (\bV_\infty - \bu_\eps) \cdot ((\bu_\eps - \bV_\infty)\cdot \nabla) \bV_\infty \dd x \dd t \\
    &\quad - \int_0^\tau \int_\Omega \bS(\nabla \bV_\infty) : \nabla (\bu_\eps - \bV_\infty) \dd x \dd t.
\end{aligned}
\end{equation}
Note that the choice $(\sigma,\bV)=(\rho_\infty, \bV_\infty)$ satisfies~\eqref{regStr}
in Theorem~\ref{thm:rei}.
In particular, we have $\nabla \bV_\infty = 0$ outside the bounded domain $B_1$, and 
$\nabla \bV_\infty \in (L^1\cap L^\infty)(\R^3)$. 
Having this in mind, we estimate
\begin{align*}
    \int_0^\tau \int_\Omega \rho_\eps (\bV_\infty - \bu_\eps) \cdot ((\bu_\eps - \bV_\infty)\cdot \nabla) \bV_\infty \dd x \dd t &\leq \|\nabla \bV_\infty\|_{L^\infty(\R^3)} \int_0^\tau \int_\Omega \rho_\eps |\bu_\eps - \bV_\infty|^2 \dd x \dd t \\
    &\lesssim \int_0^\tau \int_\Omega E_\eps(\rho_\eps, \bu_\eps | \rho_\infty, \bV_\infty) \dd x \dd t,
\end{align*}
and for any $\delta>0$,
\begin{align*}
    \int_0^\tau \int_\Omega \bS(\nabla \bV_\infty) : \nabla (\bu_\eps - \bV_\infty) \dd x \dd t &= \int_0^\tau \int_{B_1} \bS(\nabla \bV_\infty) : \nabla (\bu_\eps - \bV_\infty) \dd x \dd t \\
    &\leq C_\delta \|\nabla \bV_\infty\|_{L^2(B_1)}^2 + \delta \|\nabla(\bu_\eps - \bV_\infty)\|_{L^2((0,\tau) \times \Omega)}^2.
\end{align*}
For the first integral on the right-hand side of~\eqref{REI2} we abbreviate $\vc b = \bomega \times \bV_\infty + ((\bV_\infty - \bomega \times x)\cdot \nabla)\bV_\infty) - \vc f$ to see
\begin{align*}
    \int_0^\tau \int_\Omega \rho_\eps (\bV_\infty - \bu_\eps) \cdot \vc b \dd x \dd t = \int_0^\tau \int_\Omega [\rho_\eps]_{\ess, \eps} (\bV_\infty - \bu_\eps) \cdot \vc b \dd x \dd t + \int_0^\tau \int_\Omega [\rho_\eps]_{\res, \eps} (\bV_\infty - \bu_\eps) \cdot \vc b \dd x \dd t.
\end{align*}
Similarly to the estimate~\eqref{eq:coerciveP}, 
we see that the relative energy is coercive in the sense that
\begin{align}\label{eq:coercivity.eps}
    \frac{1}{\eps^2} \Big( [\rho_\eps]_{\reseps}^\gamma + [1]_{\reseps} + [\rho_\eps - \rho_\infty]_{\esseps}^2 \Big)
    + [\bu_\eps - \bV_\infty]_{\esseps}^2 \leq C E_\eps(\rho_\eps, \bu_\eps | \rho_\infty, \bV_\infty)
\end{align}
with a constant $C>0$ only depending on $\rho_\infty$, and thus independent of $\eps$. Note that $[\rho_\eps]_{\esseps} \leq 2\rho_\infty$ and $\vc b \in L^\infty(0,T; [L^1 \cap L^\infty] (\Omega))$ since $\vc f \in L^\infty(0,T; [L^1 \cap L^\infty] (\Omega))$. Hence, we find
\begin{align*}
    \int_0^\tau \int_\Omega [\rho_\eps]_{\esseps} (\bV_\infty - \bu_\eps) \cdot \vc b \dd x \dd t 
    &\lesssim \|\bu_\eps - \bV_\infty\|_{L^2(0,\tau; L^6(\Omega))} \|\vc b\|_{L^2(0,\tau; L^\frac65(\Omega))} 
    \\
    &\lesssim 1 + \delta \|\nabla(\bu_\eps - \bV_\infty)\|_{L^2((0,\tau) \times \Omega)}^2
\end{align*}
for any $\delta>0$.
For the residual part, we use H\"older's and Young's inequalities multiple times to conclude
\begin{align*}
    \int_\Omega [\rho_\eps]_{\reseps} (\bV_\infty - \bu_\eps) \cdot \vc b \dd x &\leq \|[\rho_\eps]_{\reseps}^{1/2} \|_{L^2(\Omega)} \|[\rho_\eps]_{\reseps}^{1/2} (\bu_\eps - \bV_\infty) \|_{L^2(\Omega)} \|\vc b\|_{L^\infty(\Omega)} \\
    &\lesssim \|[\rho_\eps]_{\reseps}\|_{L^1(\Omega)} + \|\rho_\eps |\bu_\eps - \bV_\infty|^2\|_{L^1(\Omega)} \\
    &\lesssim \|[\rho_\eps]_{\reseps}^\gamma\|_{L^1(\Omega)} + \|[1]_{\reseps}\|_{L^1(\Omega)} + \|\rho_\eps |\bu_\eps - \bV_\infty|^2\|_{L^1(\Omega)} \\
    &\lesssim \int_\Omega E_\eps(\rho_\eps, \bu_\eps | \rho_\infty, \bV_\infty) \dd x
\end{align*}
due to~\eqref{eq:coercivity.eps}
and $\eps<1$.
Choosing $\delta>0$ small enough and using the Korn-type inequality from Lemma~\ref{lem:Korn}, we are left with
\begin{align*}
    \left[ \int_\Omega E_\eps(\rho_\eps, \bu_\eps | \rho_\infty, \bV_\infty) \dd x \right]_{t=0}^{t=\tau} + \int_0^\tau \int_\Omega |\nabla (\bu_\eps - \bV_\infty)|^2 \dd x \dd t \lesssim 1 + \int_0^\tau \int_\Omega E_\eps(\rho_\eps, \bu_\eps | \rho_\infty, \bV_\infty) \dd x \dd t.
\end{align*}
By Gr\"onwall's inequality, this yields
\begin{align*}
    \sup_{t \in [0,T]} \int_\Omega E_\eps(\rho_\eps, \bu_\eps | \rho_\infty, \bV_\infty) \dd x + \int_0^T \int_\Omega |\nabla(\bu_\eps - \bV_\infty)|^2 \dd x \dd t \lesssim 1.
\end{align*}
The coercivity of the relative energy from~\eqref{eq:coercivity.eps} 
now enforces
\begin{align*}
    \|[\rho_\eps - \rho_\infty]_{\esseps}\|_{L^\infty(0,T; L^2(\Omega))}^2 + \|[\rho_\eps]_{\reseps}\|_{L^\infty(0,T; L^\gamma(\Omega))}^\gamma + \|[1]_{\reseps}\|_{L^\infty(0,T; L^1(\Omega))} &\lesssim \eps^2, \\
    \|[\bu_\eps - \bV_\infty]_{\esseps}\|_{L^\infty(0,T;L^2(\Omega))}^2 + \|\nabla(\bu_\eps - \bV_\infty)\|_{L^2(0,T; L^2(\Omega))}^2 &\lesssim 1.
\end{align*}
Returning to $\bU_\infty$, we use \eqref{estV} to see that
\begin{align*}
    \|\nabla(\bu_\eps - \bU_\infty)\|_{L^2(0,T; L^2(\Omega))} \leq \|\nabla(\bu_\eps - \bV_\infty)\|_{L^2(0,T; L^2(\Omega))} + \|\nabla(\bV_\infty - \bU_\infty)\|_{L^2(0,T; L^2(\Omega))} &\lesssim 1, \\
    \|[\bu_\eps - \bU_\infty]_{\esseps}\|_{L^\infty(0,T;L^2(\Omega))} \leq \|[\bu_\eps - \bV_\infty]_{\esseps}\|_{L^\infty(0,T;L^2(\Omega))} + \|[\bV_\infty - \bU_\infty]_{\esseps}\|_{L^\infty(0,T;L^2(\Omega))} &\lesssim 1.
\end{align*}
In conclusion, we obtain the asserted bounds.
\end{proof}

\providecommand{\weak}{\rightharpoonup}
From the uniform estimates obtained in Lemma~\ref{lem:bounds}, we can extract convergent subsequences.

\begin{Lemma}\label{lem:conv}
There exists a 
(not relabeled) subsequence of $(\rho_\eps,\bu_\eps)$ 
such that
\[
\begin{aligned}
    {[\rho_\eps - \rho_\infty]}_{\esseps} &\to 0 &&\text{ strongly in } L^\infty(0,T; L^2(\Omega)),\\
    [\rho_\eps]_{\reseps} &\to 0 &&\text{ strongly in } L^\infty(0,T; L^\gamma(\Omega)), \\
    [1]_{\reseps} &\to 0 &&\text{ strongly in } L^\infty(0,T; L^p(\Omega)),\ p\in[1,\infty), \\
    [\bu_\eps - \bU_\infty]_{\esseps} &\weak^\ast \bu - \bU_\infty &&\text{ weakly-$\ast$ in } L^\infty(0,T; L^2(\Omega)), \\
    \nabla(\bu_\eps - \bU_\infty) &\weak \nabla(\bu - \bU_\infty)&& \text{ weakly in } L^2(0,T; L^2(\Omega)),
\end{aligned}
\]
for some $\bu$ with 
$\bu-\bU_\infty\in L^2(0,T;W^{1,2}(\Omega))\cap L^\infty(0,T;L^2(\Omega))$
and $\div\bu=0$.
\end{Lemma}

\begin{proof}
If we take into account that
\[
\|[1]_{\reseps}\|_{L^\infty(0,T; L^p(\Omega))}^p = \|[1]_{\reseps}^p\|_{L^\infty(0,T; L^1(\Omega))} = \|[1]_{\reseps}\|_{L^\infty(0,T; L^1(\Omega))}
\]
for any $1 \leq p < \infty$,
the existence of a convergent subsequence in the asserted topologies follows directly from Lemma~\ref{lem:bounds}.
It only remains to show that $\bu$ is divergence free. 
To this end, we take the limit in the weak form of the continuity equation.
For any $\phi \in C_c^1((0,T) \times \Omega)$ we have 
\begin{align*}
    \int_0^T \int_\Omega (\rho_\eps - \rho_\infty) \del_t \phi + \rho_\eps(\bu_\eps - \bomega \times x) \cdot \nabla \phi \dd x \dd t = 0.
\end{align*}
For the first part, we see that
\begin{align*}
    \int_0^T \int_\Omega (\rho_\eps - \rho_\infty) \del_t \phi \dd x \dd t \lesssim \|[\rho_\eps - \rho_\infty]_{\esseps}\|_{L^\infty(0,T; L^2(\Omega))} + \|[\rho_\eps]_{\reseps}\|_{L^\infty(0,T; L^\gamma(\Omega))} + \|[1]_{\reseps}\|_{L^\infty(0,T; L^1(\Omega))}
\end{align*}
and the right-hand side vanishes as $\eps \to 0$ by the convergences obtained before. For the second term, we rewrite
\begin{align*}
    \int_0^T \int_\Omega \rho_\eps(\bu_\eps - \bomega \times x) \cdot \nabla \phi \dd x \dd t = \int_0^T \int_\Omega \rho_\eps(\bu_\eps - \bU_\infty) \cdot \nabla \phi \dd x \dd t + \int_0^T \int_\Omega \rho_\eps(\bU_\infty - \bomega \times x) \cdot \nabla \phi \dd x \dd t.
\end{align*}
Using the convergences of $[\rho_\eps - \rho_\infty]_{\esseps}$, $[\rho_\eps]_{\reseps}$, and $[1]_{\reseps}$, the last integral converges to
\begin{align*}
    \lim_{\eps \to 0} \int_0^T \int_\Omega \rho_\eps(\bU_\infty - \bomega \times x) \cdot \nabla \phi \dd x \dd t = \int_0^T \int_\Omega \rho_\infty (\bU_\infty - \bomega \times x) \cdot \nabla \phi \dd x \dd t.
\end{align*}
For the first integral, we have
\begin{align*}
    &\int_0^T \int_\Omega \rho_\eps(\bu_\eps - \bU_\infty) \cdot \nabla \phi \dd x \dd t \\
    &= \int_0^T \int_\Omega (\rho_\eps - \rho_\infty) (\bu_\eps - \bU_\infty) \cdot \nabla \phi \dd x \dd t + \rho_\infty \int_0^T \int_\Omega (\bu_\eps - \bU_\infty) \cdot \nabla \phi \dd x \dd t \\
    &= \int_0^T \int_\Omega (\rho_\eps - \rho_\infty) (\bu_\eps - \bU_\infty) \cdot \nabla \phi \dd x \dd t - \rho_\infty \int_0^T \int_\Omega \div(\bu_\eps - \bU_\infty) \phi \dd x \dd t.
\end{align*}
We obtain
\begin{align*}
    \int_0^T\int_\Omega [\rho_\eps - \rho_\infty]_{\esseps} (\bu_\eps - \bU_\infty) \cdot \nabla \phi \dd x\dd t &\lesssim \int_0^T\|[\rho_\eps - \rho_\infty]_{\esseps}\|_{L^2(\Omega)} \|[\bu_\eps - \bU_\infty]_{\esseps}\|_{L^2(\Omega)} \dd t \\
    &\lesssim \|[\rho_\eps - \rho_\infty]_{\esseps}\|_{L^\infty(0,T;L^2(\Omega))} \to 0, \\
    \int_0^T \int_\Omega [\rho_\eps - \rho_\infty]_{\reseps} (\bu_\eps - \bU_\infty) \cdot \nabla \phi \dd x \dd t &\lesssim \int_0^T \|\bu_\eps - \bU_\infty\|_{L^6(\Omega)} \Big( \|[\rho_\eps]_{\reseps}\|_{L^\frac65(\Omega)} + \|[1]_{\reseps}\|_{L^\frac65(\Omega)} \Big)\dd t \\
    &\lesssim \|[\rho_\eps]_{\reseps}^\gamma\|_{L^\infty(0,T;L^1(\Omega))} + \|[1]_{\reseps}\|_{L^\infty(0,T;L^\frac65(\Omega))} \to 0, \\
    \int_0^T \int_\Omega \div(\bu_\eps - \bU_\infty) \phi \dd x \dd t &\to \int_0^T \int_\Omega \div(\bu - \bU_\infty) \phi \dd x \dd t.
\end{align*}
Gathering the terms above, we infer
\begin{align*}
    0 &= \lim_{\eps\to0} \int_0^T \int_\Omega (\rho_\eps - \rho_\infty) \del_t \phi + \rho_\eps(\bu_\eps - \bomega \times x) \cdot \nabla \phi \dd x \dd t \\
    &= -\int_0^T \int_\Omega \rho_\infty \div(\bu - \bU_\infty) \phi \dd x \dd t + \int_0^T \int_\Omega \rho_\infty(\bU_\infty - \bomega \times x) \cdot \nabla \phi \dd x \dd t \\
    &= -\int_0^T \int_\Omega \rho_\infty \phi \div \bu \dd x \dd t
\end{align*}
for any $\phi \in C_c^1((0,T) \times \Omega)$,
leading to $\div \bu = 0$.
\end{proof}

To conclude the proof of the theorem,
we invoke the scaled REI \eqref{REI-LM} again, now with the strong solution $\bU$ of \eqref{NSE-LM-Limit}.

\begin{proof}[Proof of Theorem~\ref{thm:singLim}]
We use $(\sigma,\bV)=(\rho_\infty, \bU)$ in \eqref{REI-LM}
and utilize the identities $\div \bU = 0$, $\del_t H'(\rho_\infty) = 0$ and $\nabla H'(\rho_\infty) = 0$, together with the solution property of $(\bU, \Pi)$,
to find that
\begin{equation}\label{REI-LM-1}
\begin{aligned}
    &\left[ \int_\Omega E_\eps(\rho_\eps, \bu_\eps | \rho_\infty, \bU) \dd x \right]_{t=0}^{t=\tau} + \int_0^\tau \int_\Omega \bS(\nabla (\bu_\eps - \bU)) : \nabla (\bu_\eps - \bU) \dd x \dd t \\
    &\leq \int_0^\tau \int_\Omega \Big( \frac{\rho_\eps}{\rho_\infty} - 1 \Big) (\bU - \bu_\eps) \cdot ( \div \bS(\nabla \bU) - \nabla \Pi) \dd x \dd t \\
    &\qquad - \int_0^\tau \int_\Omega (\bU - \bu_\eps) \cdot \nabla \Pi \dd x \dd t \\
    &\qquad + \int_0^\tau \int_\Omega \rho_\eps (\bU - \bu_\eps) \cdot ((\bu_\eps - \bU) \cdot \nabla) \bU \dd x \dd t.
\end{aligned}
\end{equation}
Here we used integration by parts on the term including $\bS(\nabla\bU)$.
We decompose the second term on the right-hand side as
\begin{align*}
    \int_0^\tau \int_\Omega (\bU - \bu_\eps) \cdot \nabla \Pi \dd x \dd t = \int_0^\tau \int_\Omega [\bU - \bu_\eps]_{\esseps} \cdot \nabla \Pi \dd x \dd t + \int_0^\tau \int_\Omega [\bU - \bu_\eps]_{\reseps} \cdot \nabla \Pi \dd x \dd t =: I_1 + I_2.
\end{align*}

The integral $I_2$ vanishes as $\eps \to 0$ due to
\begin{align*}
I_2 &\leq \|\bU - \bu_\eps\|_{L^2(0,T; L^6(\Omega))} \|\nabla \Pi\|_{L^2(0,T; L^2(\Omega))} \|[1]_{\reseps}\|_{L^\infty(0,T; L^3(\Omega))} \\
&\leq \|\nabla(\bU - \bu_\eps)\|_{L^2((0,T) \times \Omega)} \|\nabla \Pi\|_{L^2(0,T; L^2(\Omega))} \|[1]_{\reseps}\|_{L^\infty(0,T; L^3(\Omega))} \to 0
\end{align*}
as $\varepsilon\to 0$
since $\nabla(\bU - \bu_\eps)$ is uniformly bounded in $L^2((0,T) \times \Omega)$ and $[1]_{\reseps} \to 0$ in $L^\infty(0,T; L^p(\Omega))$ for any $1 \leq p < \infty$
by Lemma~\ref{lem:conv}. \
For $I_1$, we use that $[\bU - \bu_\eps]_{\esseps} \weak^\ast \bU - \bu$ in $L^\infty(0,T; L^2(\Omega))$ to deduce
\begin{align*}
I_1=\int_0^\tau \int_\Omega [\bU - \bu_\eps]_{\esseps} \cdot \nabla \Pi \dd x \dd t \to \int_0^\tau \int_\Omega (\bU - \bu) \cdot \nabla \Pi \dd x \dd t = 0,
\end{align*}
which follows from $(\bU - \bu)|_{\del \Omega} = 0$ as well as $\div \bU = \div \bu = 0$.
The last term in \eqref{REI-LM-1} is clearly controlled by the relative energy $E_\eps(\rho_\eps, \bu_\eps | \rho_\infty, \bU)$
since $\nabla\bU\in L^1(0,T;L^\infty(\Omega))$. 
Lastly, splitting the first term on the right-hand side of \eqref{REI-LM-1} into its essential and residual part, 
we can derive estimates analogous to~\eqref{eq:reldiff.res} and~\eqref{eq:reldiff.ess}. 
Due to the integrability assumptions from~\eqref{eq:singLim.assstrong},
we finally find
\begin{align*}
    &\left[ \int_\Omega E_\eps(\rho_\eps, \bu_\eps | \rho_\infty, \bU) \dd x \right]_{t=0}^{t=\tau} + \int_0^\tau \int_\Omega \bS(\nabla (\bu_\eps - \bU)) : \nabla (\bu_\eps - \bU) \dd x \dd t \\
    &\lesssim \int_0^\tau \int_\Omega E_\eps(\rho_\eps, \bu_\eps | \rho_\infty, \bU) \dd x \dd t + o_{\eps},
\end{align*}
where $o_\eps\geq 0$ is a scalar with $o_\eps\to 0$ as $\eps\to0$.
Gr\"onwall's inequality then yields
\begin{align*}
    &\sup_{t \in [0, T]} \int_\Omega E_\eps(\rho_\eps, \bu_\eps | \rho_\infty, \bU) \dd x + \int_0^T \int_\Omega \bS(\nabla (\bu_\eps - \bU)) : \nabla (\bu_\eps - \bU) \dd x \dd t \\
    &\lesssim \int_\Omega E_\eps(\rho_{\eps, 0}, \bu_{\eps, 0} | \rho_\infty, \bU_0) \dd x + o_{\eps}.
\end{align*}
Due to assumption \eqref{ass:initDat},
we now conclude the limit~\eqref{eq:limit.lm} 
as $\varepsilon\to 0$
along the subsequence chosen in Lemma~\ref{lem:conv}.
However, since this limit is independent of the chosen subsequence,
the sequence itself converges, which finishes the proof of Theorem~\ref{thm:singLim}.
\end{proof}

\subsection{Quantitative estimates, proof of Theorem~\ref{thm:singLimRate}}
Lastly, we derive the convergence rates as given in Theorem~\ref{thm:singLimRate}. To this end, we recall the uniform bounds obtained in Lemma~\ref{lem:bounds} and the inequality \eqref{REI-LM-1}, which we rewrite in the form
\begin{equation}\label{REI-LM-R1}
\begin{aligned}
    &\left[ \int_\Omega E_\eps(\rho_\eps, \bu_\eps | \rho_\infty, \bU) \dd x \right]_{t=0}^{t=\tau} + \int_0^\tau \int_\Omega \bS(\nabla (\bu_\eps - \bU)) : \nabla (\bu_\eps - \bU) \dd x \dd t \\
    &\leq \int_0^\tau \int_\Omega \Big( \frac{\rho_\eps}{\rho_\infty} - 1 \Big) (\bU - \bu_\eps) \cdot \div \bS(\nabla \bU) \dd x \dd t \\
    &\qquad - \int_0^\tau \int_\Omega \frac{\rho_\eps}{\rho_\infty} (\bU - \bu_\eps) \cdot \nabla \Pi \dd x \dd t \\
    &\qquad + \int_0^\tau \int_\Omega \rho_\eps (\bU - \bu_\eps) \cdot ((\bu_\eps - \bU) \cdot \nabla) \bU \dd x \dd t.
\end{aligned}
\end{equation}

\begin{proof}[Proof of Theorem~\ref{thm:singLimRate}]
As before, the last term on the right-hand side of~\eqref{REI-LM-R1} is estimated by the relative energy itself. The first term we split into its essential and residual part to see
\begin{align*}
    &\int_0^\tau \int_\Omega \Big[ \frac{\rho_\eps}{\rho_\infty} - 1 \Big]_{\ess, \eps} (\bU - \bu_\eps) \cdot \div \bS(\nabla \bU) \dd x \dd t \\
    &\quad \leq C \|[\rho_\eps - \rho_\infty]_{\ess, \eps} \|_{L^\infty(0,T; L^2(\Omega))} \|\bU - \bu_\eps \|_{L^2(0,T; L^6(\Omega))} \|\div \bS(\nabla \bU)\|_{L^2(0,T; L^3(\Omega))} \leq C \eps, \\
    &\int_0^\tau \int_\Omega \Big[ \frac{\rho_\eps}{\rho_\infty} - 1 \Big]_{\res, \eps} (\bU - \bu_\eps) \cdot \div \bS(\nabla \bU) \dd x \dd t \\
    &\quad \leq C \|[\rho_\eps - \rho_\infty]_{\res, \eps} \|_{L^\infty(0,T; L^\gamma(\Omega))} \|\bU - \bu_\eps \|_{L^2(0,T; L^6(\Omega))} \|\div \bS(\nabla \bU)\|_{L^2(0,T; L^\frac{6\gamma}{5\gamma-6}(\Omega))} \leq C \eps^\frac2\gamma,
\end{align*}
where we used Sobolev embeddings and the uniform bounds obtained in Lemma~\ref{lem:bounds}. 
For the remaining pressure part, we make use 
of the function $\bV_\infty$ defined in~\eqref{eq:V.def} and infer
\begin{align*}
    &\int_0^\tau \int_\Omega \frac{\rho_\eps}{\rho_\infty} (\bU - \bu_\eps) \cdot \nabla \Pi \dd x \dd t 
    \\
    &\quad
    = \int_0^\tau \int_\Omega \Big( \frac{\rho_\eps}{\rho_\infty} - 1\Big) (\bU-\bV_\infty) \cdot \nabla \Pi 
    - \frac{\rho_\eps}{\rho_\infty} (\bu_\eps-\bV_\infty) \cdot \nabla \Pi 
    + (\bU-\bV_\infty) \cdot \nabla \Pi   
    \dd x \dd t\\
    &\quad
    = \int_0^\tau \int_\Omega \Big( \frac{\rho_\eps}{\rho_\infty} - 1\Big) (\bU-\bV_\infty) \cdot \nabla \Pi \dd x \dd t 
    - \frac{1}{\rho_\infty} \int_0^\tau \int_\Omega\rho_\eps (\bu_\eps-\bV_\infty) \cdot \nabla \Pi \dd x \dd t,
\end{align*}
where we used that the integral over $(\bU-\bV_\infty) \cdot \nabla \Pi$ vanishes by the 
divergence theorem due to
$\div \bU=\div\bV_\infty=0$ in $\Omega$ and $\bU=\bV_\infty=\bomega\times x$ on $\partial\Omega$. 
For the second integral, we deduce from the continuity equation \eqref{eq:cont.weak} that 
\begin{align*}
    - \int_0^\tau \int_\Omega\rho_\eps (\bu_\eps-\bV_\infty) \cdot \nabla \Pi \dd x \dd t &= \int_0^\tau \int_\Omega (\rho_\eps - \rho_\infty) \del_t \Pi + \rho_\eps (\bV_\infty-\bomega \times x) \cdot \nabla \Pi \dd x \dd t  \\
    &\qquad 
    -\left[ \int_\Omega (\rho_{\eps} - \rho_\infty) \Pi\dd x \right]_{t=0}^{t=\tau},
\end{align*}
which follows by an approximation that requires to introduce the 
additional terms related to $\rho_\infty$, which can be carried out similarly to the approximation argument in the proof of Theorem~\ref{thm:rei}. Note that $\Pi$ is an admissible test function due to the regularity assumed in~\eqref{eq:sigLimRate.pressurereg}.
Moreover, the divergence theorem yields
\[
\int_\Omega(\bV_\infty-\bomega \times x) \cdot \nabla \Pi \dd x = 0 .
\]
Hence, we may write
\begin{align*}
    \int_0^\tau \int_\Omega \frac{\rho_\eps}{\rho_\infty} (\bU - \bu_\eps) \cdot \nabla \Pi \dd x \dd t 
    &= \int_0^\tau \int_\Omega \Big( \frac{\rho_\eps}{\rho_\infty} - 1\Big) (\bU - \bV_\infty) \cdot \nabla \Pi \dd x \dd t \\
    &\qquad + \int_0^\tau \int_\Omega \Big( \frac{\rho_\eps}{\rho_\infty} - 1 \Big) (\bV_\infty - \omega \times x) \cdot \nabla \Pi \dd x \dd t \\
    &\qquad + \int_0^\tau \int_\Omega \Big( \frac{\rho_\eps}{\rho_\infty} - 1 \Big) \del_t \Pi \dd x \dd t -\left[\int_\Omega  \Big(\frac{\rho_{\eps}}{\rho_\infty} - 1 \Big) \Pi \dd x\right]_{t=0}^{t=\tau}.
\end{align*}
Splitting as before in essential and residual part, and using that
\begin{align*}
    \left[\int_\Omega  \Big(\frac{\rho_{\eps}}{\rho_\infty} - 1 \Big) \Pi \dd x\right]_{t=0}^{t=\tau} \leq C \|\rho_\eps - \rho_\infty\|_{L^\infty(0,T; [L^2+L^\gamma] (\Omega))} \|\Pi\|_{L^\infty(0,T; [L^2 \cap L^\frac{\gamma}{\gamma-1}] (\Omega))} \leq C (\eps + \eps^\frac2\gamma)
\end{align*}
due to~\eqref{eq:sigLimRate.pressurereg},
we thus infer
\begin{align*}
    \int_0^\tau \int_\Omega \frac{\rho_\eps}{\rho_\infty} (\bU - \bu_\eps) \cdot \nabla \Pi \dd x \dd t \leq C(\eps + \eps^\frac2\gamma).
\end{align*}
Back to \eqref{REI-LM-R1}, we find
\begin{align*}
    &\left[ \int_\Omega E_\eps(\rho_\eps, \bu_\eps | \rho_\infty, \bU) \dd x \right]_{t=0}^{t=\tau} + \int_0^\tau \int_\Omega \bS(\nabla (\bu_\eps - \bU)) : \nabla (\bu_\eps - \bU) \dd x \dd t \\
    &\leq C(\eps + \eps^\frac2\gamma) + C \int_0^\tau \int_\Omega E_\eps(\rho_\eps, \bu_\eps | \rho_\infty, \bU) \dd x \dd t,
\end{align*}
so that Gr\"onwall's inequality finally yields
\begin{align*}
    &\sup_{\tau \in [0,T]} \int_\Omega E_\eps(\rho_\eps, \bu_\eps | \rho_\infty, \bU)(\tau) \dd x + \int_0^T\int_\Omega \bS(\nabla (\bu_\eps - \bU)) : \nabla (\bu_\eps - \bU) \dd x \dd t \\
    &\leq C \Big( \int_\Omega E_\eps(\rho_{\eps, 0}, \bu_{\eps, 0} | \rho_\infty, \bU_0) \dd x + \eps + \eps^\frac2\gamma \Big).
\end{align*}
We finish the proof of Theorem~\ref{thm:singLimRate} by applying Korn's inequality from Lemma~\ref{lem:Korn}.
\end{proof}

\subsection*{Acknowledgments}
T.~Eiter’s research has been funded by Deutsche Forschungsgemeinschaft (DFG) through grant CRC 1114 ``Scaling Cascades in Complex Systems'', Project Number 235221301, Project YIP. F.~Oschmann has been supported by the Czech Science Foundation (GA\v CR) project 22-01591S, and the Czech Academy of Sciences project L100192351.
\v S. Ne\v casov\' a was supported by Premium Academia of \v S.N. The Institute of Mathematics, CAS is supported by RVO:67985840.

\printbibliography

\end{document}